\documentclass[12pt]{amsart}
\usepackage{comment}
\usepackage[utf8]{inputenc} 
\usepackage[T1]{fontenc} 

\usepackage[dvipsnames]{xcolor}

\usepackage{amsfonts}
\usepackage{amssymb}
\usepackage{amsthm}
\usepackage{amsmath}
\usepackage{amscd}
\usepackage[shortlabels]{enumitem}
\usepackage{mathrsfs}
\usepackage{tikz,pgfplots}
\usetikzlibrary{calc,arrows,decorations.pathreplacing}

\usepackage{nicefrac, xfrac}
\usepackage{mathtools}
\usepackage[bottom]{footmisc}
\usepackage[mathcal]{euscript}
\usepackage{float}

\usepackage{wrapfig}
\usepackage[rightcaption]{sidecap}
\sidecaptionvpos{figure}{c} 
\usepackage[font=small,skip=5pt]{caption}

\usepackage[pagebackref, colorlinks = true,
linkcolor = red,
urlcolor  = blue,
citecolor = blue,
anchorcolor = blue]{hyperref}

\theoremstyle{plain}

\newtheorem{theorem}{Theorem}[section]

\newtheorem{lemma}[theorem]{Lemma}
\newtheorem*{thm}{Main Theorem}

\newtheorem{proposition}[theorem]{Proposition}

\theoremstyle{definition}
\newtheorem{definition}[theorem]{Definition}

\newtheorem{remark}[theorem]{Remark}
\newtheorem{example}[theorem]{Example}

\renewcommand{\phi}{\varphi}
\renewcommand{\epsilon}{\varepsilon}

\newcommand{\set}[1]{\ensuremath{ \left\{#1\right\} }} 
\newcommand{\norm}[1]{\ensuremath{ \left\lVert#1\right\rVert }} 
\newcommand{\absval}[1]{\ensuremath{ \left\lvert#1\right\rvert }} 

\newcommand{\bs}{\ensuremath{\backslash}}

\newcommand{\lt}{\ensuremath{\left}}
\newcommand{\rt}{\ensuremath{\right}}

\newcommand{\Om}{\ensuremath{\Omega}}
\newcommand{\Lm}{\ensuremath{\Lambda}}

\newcommand{\Dl}{\ensuremath{\Delta}}

\newcommand{\Ta}{\ensuremath{\Theta}}

\newcommand{\om}{\ensuremath{\omega}}
\newcommand{\lm}{\ensuremath{\lambda}}
\newcommand{\gm}{\ensuremath{\gamma}}  
\newcommand{\al}{\ensuremath{\alpha}} 
\newcommand{\bt}{\ensuremath{\beta}}  
\newcommand{\sg}{\ensuremath{\sigma}}
\newcommand{\ep}{\ensuremath{\epsilon}}

\DeclareSymbolFont{bbold}{U}{bbold}{m}{n}
\DeclareSymbolFontAlphabet{\mathbbold}{bbold}

\newcommand{\ind}{\ensuremath{\mathbbold{1}}}

\DeclareMathOperator{\diam}{diam}
\DeclareMathOperator{\supp}{supp}

\newcommand{\cC}{\ensuremath{\mathcal{C}}}

\newcommand{\cE}{\ensuremath{\mathcal{E}}}

\newcommand{\cL}{\ensuremath{\mathcal{L}}}

\newcommand{\cP}{\ensuremath{\mathcal{P}}}

\newcommand{\cZ}{\ensuremath{\mathcal{Z}}}

\newcommand{\CC}{\ensuremath{\mathbb C}} 

\newcommand{\NN}{\ensuremath{\mathbb N}}

\newcommand{\RR}{\ensuremath{\mathbb R}}

\usepackage{mathtools}

\newcommand{\maeom}{$m$ a.e. $\om\in\Om$}

\newcommand{\nothing}{}
\newcommand{\IO}{I} 
\DeclareMathOperator{\Einf}{Einf} 
\DeclareMathOperator{\Esup}{Esup}

\def\var{\text{{\rm var}}} 
\def\BV{\text{{\rm BV}}}

\def\leb{\text{{\rm Leb}}}

\def\lt{\left}
\def\rt{\right}

\providecommand{\phantomsection}{}
\AtBeginDocument{\let\textlabel\label}
\makeatletter
\newcommand{\mylabel}[2]{\raisebox{.7\normalbaselineskip}{\phantomsection}(#1)%
	\def\@currentlabel{#1}\textlabel{#2}}
\makeatother

\newcommand{\cEP}{\ensuremath{\cE P}} 
\allowdisplaybreaks

\numberwithin{equation}{section}
\title
[Equilibrium states for non-transitive random dynamical systems]
{Equilibrium states for non-transitive random open and closed dynamical systems}
\date{\today}

\author{Jason Atnip}
\address{School of Mathematics and Statistics, University of New South Wales, Sydney, NSW 2052, Australia}
\email{\href{j.atnip@unsw.edu.au}{j.atnip@unsw.edu.au} }

\author{Gary Froyland}
\address{School of Mathematics and Statistics, University of New South Wales, Sydney, NSW 2052, Australia}
\email{\href{g.froyland@unsw.edu.au}{g.froyland@unsw.edu.au} }

\author{Cecilia Gonz\'alez-Tokman}
\address{School of Mathematics and Physics, The University of Queensland, St Lucia, QLD 4072, Australia}
\email{\href{cecilia.gt@uq.edu.au}{cecilia.gt@uq.edu.au} }

\author{Sandro Vaienti}
\address{Aix Marseille Université, Université de Toulon, CNRS, CPT, 13009 Marseille, France}
\email{\href{vaienti@cpt.univ-mrs.fr}{vaienti@cpt.univ-mrs.fr} }

\begin{document}
\begin{abstract}
	We prove a 
	random Ruelle--Perron--Frobenius theorem and the existence of relative equilibrium states 
	for a class of random open and closed interval maps, without imposing transitivity requirements, such as mixing and covering conditions, which are prevalent in the literature. 
	This theorem provides existence and uniqueness of random conformal and invariant measures with exponential decay of correlations, and allows us to expand the class of examples of (random) dynamical systems amenable to multiplicative ergodic theory and the thermodynamic formalism. Applications include open and closed  non-transitive random maps, and a connection between Lyapunov exponents and escape rates through random holes.
	We are also able to treat 
	random intermittent maps with geometric potentials.
\end{abstract}
\maketitle
\tableofcontents
\section{Introduction}

Non-autonomous or random dynamical systems provide  flexible mathematical models to analyze a wide range of forced and noisy phenomena. They have been  identified as an important direction going forward in the study of chaotic systems \cite{Young13}.
One of the obstacles in the investigation of the long-term properties of such systems stems from the difficulty in identifying concrete examples for which the available theoretical results apply. 
This work uncovers scenarios where ergodic-theoretical tools can be used to establish results related to the thermodynamic formalism and decay of correlations for random dynamical systems, without imposing requirements such as transitivity or covering, which are often difficult to verify in this context.

For autonomous (time-homogeneous) finite-state Markov chains and systems whose dynamics can be encoded by
them,
such as shifts of finite type and systems with a Markov partition, 
one can use normal forms for reducible matrices \cite[Vol. 2]{Gantmacher} to analyze the dynamics using irreducible components as building blocks.
In sharp contrast,
there is no available decomposition of non-autonomous (random) systems into transitive or irreducible components. For instance, Buzzi \cite[\S 0.2]{BuzziEDC} noted difficulties in decomposing one-dimensional piecewise expanding random systems into
\emph{pathwise irreducible components}, and hence in the search for decompositions that could play the role of normal forms  in this setting. 
Accordingly, the study of decay of correlations and Ruelle--Perron--Frobenius type results in the random  setting has so far relied on stronger hypotheses, such as mixing and/or covering conditions \cite{Bogenschutz_RuelleTransferOperator_1995a,BaladiCorrelationSpectrum, BuzziEDC,kifer_thermodynamic_2008,mayer_distance_2011,MayerUrbanski15,DFGTV-cmp,atnip_critically_2020,Horan1,AFGTV1,stadlbauer_suzuki_varandas_2020,AFGTV2}. Similar assumptions appear in the investigation of memory loss in time-dependent systems \cite{OttStenlundYoung,StenlundYoungZhang,GuptaOttTorok13,DemersLiverani21}.

In this work,  we exhibit new  examples of random dynamical systems
for which invariant measures (relative equilibrium states) with exponential decay of correlations can be constructed.
We do not impose transitivity assumptions -- so neither topological mixing nor covering conditions are assumed -- but instead require that the random maps and random potentials satisfy a contracting type condition, on average; see Definition~\ref{def:cp} for details.
Naturally, when such results hold, one expects to obtain a one-dimensional top equivariant direction for the (random) transfer operator.  
Indeed, under mild extra assumptions, we also show that the multiplicative ergodic theorem of Froyland, Lloyd and Quas \cite{FLQ2} applies in this setting and yields a unique random Ruelle--Perron--Frobenius decomposition and further information.
Our approach builds on the concept of a contracting potential, introduced in the autonomous setting by Liverani, Saussol and Vaienti  \cite{liverani_conformal_1998}, but we work with random cones of functions, conveniently defined in terms of (essential) infimum and variation.
This work may also be regarded as a generalisation, complementary to \cite{AFGTV2}, of the work of Liverani and Maume-Deschamps \cite{liverani_maume_2003} to the random setting. Furthermore, our approach allows us prove results for both open and closed settings simultaneously, in a concise manner.

Our main results may be summarized as follows. See \S\ref{S:setting} for the allowed class of random open (and closed) maps, Definition~\ref{def:cp} for the notion of strongly contracting potential and \S\ref{S:pfMainThm} for precise statements and proofs. For the related random Ruelle--Perron--Frobenius type decomposition, see Theorem~\ref{thm:randompf}. 
Throughout this work, $\Einf(f)$ is the essential infimum of $f$ with respect to the Lebesgue measure. 

\begin{thm}
	Let $\cL_\om$ be the transfer operator associated to a random strongly contracting potential for a random open (or closed)  map of the interval $\{(T_\om,H_\om)\}_{\om\in\Om}$ (or $\{T_\om\}_{\om\in\Om}$), driven by an ergodic, invertible, probability preserving transformation $\sigma: (\Om, m) \to (\Om, m)$. 
	Then, there exist equivariant families, $\{q_\om\}_{\om\in\Om}$ and $\{\nu_\om\}_{\om\in\Om}$, of bounded variation functions and probability measures, respectively, given by 
	\[
	q_\om={\nothing}\lim_{n\to\infty}\frac{\cL_{\sg^{-n}\om}^{(n)} 1}{ \Einf(\cL_{\sg^{-n}\om}^{(n)} 1)}
	\quad \text{and} \quad
	\nu_\om (\cdot) =\lim_{n\to\infty}\frac{\Einf_{\nothing} (\cL_{\om}^{(n)} (\cdot))}{\Einf(\cL_{\om}^{(n)} 1)},
	\]
	such that
	$\cL_\om q_\om = \lm^-_\om q_{\sg\om}$ and  $\nu_\om(\cdot) = \lm^+_\om \nu_{\sg\om}(\cL_\om (\cdot))$, with
	$\int \log \lm_\om^+dm= \int \log \lm_\om^-dm$. The multipliers\footnote{It will be shown that
		$\lm_\om^-= \frac{ \nu_\om(q_\om)\lm_\om^+ }{ \nu_{\sg\om}(q_{\sg\om})}$, see \eqref{eq:lamnu}.} $\{\lm^\pm_\om\}_{\om\in\Om}$ also satisfy \eqref{eq:equiv-} and \eqref{eq:confmeas}.
	Define $\mu_\om$ by $\int f d\mu_\om := \frac{\int f q_\om d\nu_\om}{\nu_\om(q_\om)}$. Then, 
	\begin{equation*}
		\int f   d\mu_{\sg \om} =  \int f \circ T_\om d\mu_\om,
	\end{equation*}
	and $\{\mu_\om\}_{\om\in\Om}$ yields the unique relative equilibrium state for the system.
	Furthermore, there exist   $0<r<1$ and a measurable, tempered\footnote{A function $a:\Om \to \RR$ is \emph{tempered} if for \maeom, $\lim_{|n|\to\infty} \frac1n \log |a(\sg^n\om)|=0$.  Equivalently, for every $\ep>0$ there exists $A_\om>0$ such that  for every $n\in \NN$, $a(\sg^n\om)\leq A_\om e^{\ep|n|}$.} $C_\om>0$ such that for every $f\in L^1(\nu_\om)$, $\~f\in L^1(\nu_{\sg^n\om})$, and $h \in BV$, 
	\begin{equation*}
		\big| \mu_{\sg^{-n}\om}(f\circ  T^{(n)}_{\sg^{-n}\om} \cdot h)  - \mu_{\om}(f)\mu_{\sg^{-n}\om}(h)\big|
		\leq C_\om\|f\|_{L^1(\nu_{\om})} \|h\|_{BV} r^n, \quad \text{and}
	\end{equation*}
	\begin{equation*}
		\big| \mu_\om(\~f\circ  T^{(n)}_\om \cdot h)  - \mu_{\sg^n\om}(\~f) \mu_{\om}(h)\big|
		\leq C_\om\|\~f\|_{L^1(\nu_{\sg^n\om})} \|h\|_{BV} r^n.
	\end{equation*}
\end{thm}

In \S\ref{S:nonCovering},
we show that our results indeed apply to  non-transitive, non-mixing and non-covering maps;  see Figure~\ref{fig:ncmap} and Example~\ref{ex:nonmixing}. This is not a trivial example because, depending on the potential, the random invariant measures may or may not be supported inside the invariant interval around 1/2.
As a special case, 
we also show (Lemma~\ref{lem:contracting4geompot}) that when the geometric potential $-\log |T'_\om|$ is strongly contracting, the random map is in fact covering. 
In particular, $-\log |T'_\om|$ is not strongly contracting in Example~\ref{ex:nonmixing}.
Our results also apply to open and closed   random intermittent maps (\S\ref{S:intermittent}), and allow us to investigate escape rates for random open systems (\S\ref{S:escapeRates}).

\begin{figure}[ht]
	\centering	
	\begin{tikzpicture}
		\begin{axis}[clip=false,
			xmin=0,
			xmax=1.001,samples=351,domain=0:1,
			ymin=0,
			ymax=1.001,
			xtick={.1,.2,.25,.3,.4,.5,.6,.7,.8,.9},
			xticklabels={},
			extra x ticks={0,1},
			extra x tick labels={0,1},
			ytick={1},
			every axis plot/.append style={ultra thick}
			]
			\addplot[domain=0:1, dashed] {x};
			
			\draw[dashed] (40,40) rectangle (60,60); 
			\addplot[domain=.4:.5,blue] {2*x-.4};
			\addplot[domain=.5:.6,blue] {-2*x+1.6};

			\addplot[domain=0:.1,blue] {1-10*x}; 
			\addplot[domain=.1:.2,blue] {1/.06*x*(x-.1)};
			\addplot[domain=.25:.3,blue] {1/.06*x*(x-.1)};
			\addplot[domain=.3:.4,blue] {10*(x-.3)};
			\addplot[domain=.6:.7,blue] {-10*(x-.6)+1};
			\addplot[domain=.7:.8,blue] {(1/(.1^(3/4)))*(x-.7)^(3/4)};
			\addplot[domain=.9:1,blue] {(1/(.1^(3/4)))*(1-x)^(3/4)};

			\draw [line width=4pt, red] (20,0) -- (25,0);
			\draw [line width=1pt, red, dashed] (20,0) -- (20,100);
			\draw [line width=1pt, red, dashed] (25,0) -- (25,100);			
			\draw [line width=4pt, red] (80,0) -- (90,0);
			\draw [line width=1pt, red, dashed] (80,0) -- (90,0);
			\draw [line width=1pt, red, dashed] (80,0) -- (80,100);
			\draw [line width=1pt, red, dashed] (90,0) -- (90,100);
			
			\draw[decorate,decoration={brace,amplitude=8pt,mirror}] (20,-3)  -- (25,-3) ; 
			\draw[decorate,decoration={brace,amplitude=8pt,mirror}] (80,-3)  -- (90,-3) ; 
			\draw[decorate,decoration={brace,amplitude=8pt,mirror}] (22.5,-8)  -- (85,-8) ; 
			\node at (53.75, -18){$H_\omega$};
		\end{axis}
	\end{tikzpicture}\caption{A non-transitive open map}\label{fig:ncmap}
\end{figure}
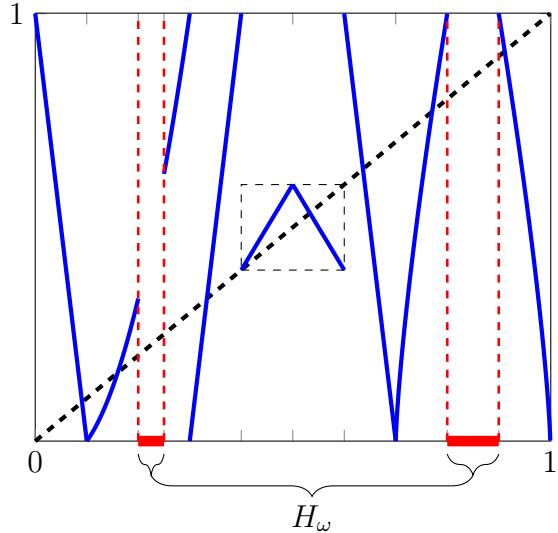

In contrast to previous works requiring the identification of a (random) conformal measure first,  our approach decouples the construction
of equivariant densities, $q_\om$, and conformal measures, $\nu_\om$, and builds these dual objects in a symmetric fashion.
In short, densities depend on the past, while measures depend on the future.
An extra element arising in the random setting is that, unlike in the autonomous case, the forward and backward multipliers $\lm_\om^\pm$ arising from these constructions are not necessarily equal, and so the densities may not be normalized with respect to the conformal measures. Thus, to find a (random) invariant measure $\mu_\om$, one should normalize: i.e.\  $\mu_\om=\frac{q_\om \nu_\om}{\nu_\om(q_\om)}$. 

This work complements previous works of the authors \cite{AFGTV1,AFGTV2},
where they have developed a general thermodynamic formalism for random open and closed dynamical systems, without the strongly contracting assumption of this work, but imposing covering type conditions.
The present approach also incorporates 
the use of a random family of cones, a strategy previously used in \cite{kifer_thermodynamic_2008}, the references therein, and recently in  \cite{stadlbauer_suzuki_varandas_2020}.

\section{Notation and setting}\label{S:setting}

The following notation will be used throughout the paper.
Let $\IO\subset \RR$ be a compact interval. For $Z\subset \IO$, we denote by $\Einf_Z (f)$ the \emph{essential infimum} of $f$ on $Z$, with respect to the Lebesgue measure. We also write $\Einf(f)$ instead of $\Einf_{\IO} (f)$, and define $\Einf_{\emptyset} (f)=0$. Similar conventions apply to the \emph{essential supremum}  $\Esup_Z(f)$. 
Let the \emph{variation} of $f$ on $Z$ be
$\text{var}_{Z}(f)=\underset{x_0<\dots <x_k, \, x_j\in Z}{\sup}\sum_{j=0}^{k-1}\absval{f(x_{j+1})-f(x_j)}$, and  $\var(f):=\var_{\IO}(f)$.
Let $BV\subset L^\infty(\leb)$ 
be the set of (equivalence classes of) functions of \emph{bounded variation} on $\IO$, with norm
$\|f\|_{BV}:=\underset{\tilde{f}=f \ \text{Leb a.e.}}{\inf} \text{var}_{\IO}(\tilde{f})+\|f\|_\infty$, where $\|f\|_\infty:= \Esup(f)$.
It follows from Rychlik \cite{rychlik_bounded_1983} that BV is a Banach space, and that
if $f$ is a function of bounded variation, then it is always possible to choose a representative of minimal variation.
From now on, we will work with such representatives, and will no longer distinguish between functions of bounded variation and their equivalence classes in BV.
Furthermore, we recall that  two functions of bounded variation $f, \tilde{f}:\IO \to \RR$  coincide Lebesgue almost everywhere if and only if the values of $f$ and $\tilde{f}$ differ in an at most countable set. Thus, if two BV functions coincide Lebesgue almost everywhere, then they are also equivalent with respect to any other non-atomic measure.

Let $$T_1, T_2, \dots: \IO \to \IO$$ be a countable collection of 
\emph{maps} such that for each $j\in \NN$ there exists a finite partition of $I \pmod{\text{Leb}}$  such that $T_j$ is monotonic and continuous on each atom.
Let   
$$H_1, H_2, \dots \subset \IO$$ be such that for each $j\in \NN$, $H_j\subset \IO$ is a (possibly empty) finite union of intervals, called \emph{holes}.
Assume\footnote{This assumption rules out the possibility of \emph{periodicity}, and
	is used to control infima in our arguments.} that for every $j \in \NN$ there is at least one full branch of $T_j$ completely contained in  $X_j:=\IO\setminus H_j$. 
Consider \emph{weights} of bounded variation 
$$g_1, g_2, \dots :\IO \to \RR^+, \quad j=1,2,\dots,$$
with associated \emph{potentials}  $\phi_j:=\log g_j$. 

Let $(\Om, m)$ be a complete probability space,
and $\sg:(\Om, m)\to (\Om, m)$ be an ergodic, invertible, probability preserving transformation, called the \emph{driving system}.
Let $\Om=\cup_{j=1}^\infty \Om_j$ be an (at most) countable partition of $\Om$ into measurable sets. For each $\om \in \Om_j$, let $T_\om=T_j, H_\om=H_j, X_\om=X_j, g_\om=g_j$. 
These assumptions ensure the quantities involved in the definition of strongly contracting potential (Definition~\ref{def:cp}) are measurable.
We refer to $\{(T_\om,H_\om)\}$ as a \emph{random open map},
and to $\{T_\om\}$ (or $\{(T_\om, \emptyset) \}$) as a \emph{random closed map}.

For each $\om\in\Om$ and $n\in \NN$,  let $T_\om^{(n)}:=T_{\sg^{n-1}\om} \circ \dots \circ T_{\sg \om} \circ T_\om$, $T_\om^{(0)}:=Id$, and $g_\om^{(n)}:=g_{\sg^{n-1}\om}  \dots  g_{\sg \om}  g_\om$.  Let $\cZ_{\om}^{(n)}$ be the monotonicity partition of $T_\om^{(n)}$, and
$\mathring{\cZ}_{\om}^{(n)}$ be the coarsest partition of the survivor set $X_{\om,n}:=\cap_{j=0}^{n-1} (T_{\om}^{(j)})^{-1} (X_{\sg^j\om})$ into intervals, such that for each $Z\in\mathring{\cZ}_{\om}^{(n)}$  there exists $Z'\in\cZ_{\om}^{(n)}$ such that $Z\subset Z'$.
We  split $\mathring{\cZ}_{\om}^{(n)}$ into $\mathring{\cZ}_{\om,f}^{(n)}$ and $\mathring{\cZ}_{\om,p}^{(n)}$, corresponding to the full and non-full (or partial) branches of $T_\om^{(n)}|_{X_{\om,n}}$. That is, $Z\in \mathring{\cZ}_{\om,f}^{(n)}$ if and only if $T_\om^{(n)}(Z)=\IO$. A collection of intervals $Z_1, \dots, Z_k \in \mathring{\cZ}_{\om,p}^{(n)}$ is said to be a collection of \textit{contiguous non-full intervals} for $T_\om^{(n)}$ (or, more precisely of $(T_\om^{(n)}, H_{\om,n})$, where $H_{\om,n}:=\IO\setminus X_{\om,n}$) if there is no element of $\mathring{\cZ}_{\om,f}^{(n)}$ in between them\footnote{This condition has been considered in \cite[\S6]{liverani_maume_2003}.}; that is, if  the convex hull of $\cup_{j=1}^k Z_j$ does not contain any element of $\mathring{\cZ}_{\om,f}^{(n)}$. 
We denote by $b_{\om,f}^{(n)}$ the cardinality of $\mathring{\cZ}_{\om,f}^{(n)}$ and by $\xi_{\om}^{(n)}$ the largest number of contiguous
non-full (or partial) intervals for $T_\om^{(n)}$.

The transfer operator for the random (open or closed) map $\{(T_\om, H_\om)\}_{\om \in \Om}$ with potential $\{\log g_\om\}_{\om \in \Om}$\footnote{In the sequel, we will exclude the sub-index $\om\in \Om$ from the notation, and write e.g. $\{\log g_\om\}$.}, acting on $f\in BV$ is defined by:
\[ \cL_{\om}f = \sum_{Z\in\mathring{\cZ}_{\om}^{(1)}} \ind_{T_\om(Z)} \lt((f g_{\om})\circ T_{\om,Z}^{-1}\rt),\]
where 	$	
T_{\om,Z}^{-1}:T_\om(Z)\to Z
$ 
is the inverse of $T_\om|_Z$.
Its $n$ step iteration, $\cL_{\om}^{(n)}f:=\cL_{\sg^{n-1}\om} \circ \dots \circ \cL_{\sg \om} \circ \cL_\om$,	is given by
\[ \cL_{\om}^{(n)}f = \sum_{Z\in\mathring{\cZ}_{\om}^{(n)}} \ind_{T_\om^{(n)}(Z)} \lt((f g_{\om}^{(n)})\circ T_{\om,Z}^{-n}\rt),\]
where 	$	
T_{\om,Z}^{-n}:T_\om^{(n)}(Z)\to Z
$ 
is the inverse of $T_\om^{(n)}|_Z$.

\section{Basic estimates}
The estimates in this section generalize arguments developed in  \cite{liverani_maume_2003}. 
\subsection{Infimum estimates} A direct estimate yields,  for every $\om\in\Om$,  $f\in\BV$, and  $n\in\NN$,
\begin{align*}
	\sum_{Z\in \mathring{\cZ}_{\om,f}^{(n)}} \Einf_Z|f| &\leq b_{\om,f}^{(n)} \big(\var(f) + \Einf(|f|) \big).
\end{align*}
By comparing the infimum over $Z\in \mathring{\cZ}_{\om,p}^{(n)}$ with the infimum over its closest full-branch neighbor, one gets
\begin{align}\label{eq:sumNF}
	\sum_{Z\in \mathring{\cZ}_{\om,p}^{(n)}} \Einf_Z|f| &\leq 2\xi_{\om}^{(n)} \left(\var(f)+\sum_{Z\in\mathring{\cZ}_{\om,f}^{(n)}}\Einf_Z|f| \right).
\end{align}
Furthermore, if $f\geq0$,
\begin{align}\label{eq:infLf}
	\Einf({\cL_{\om}^{(n)}f}) \geq \sum_{Z\in \mathring{\cZ}_{\om,f}^{(n)}} \Einf_Z (g_{\om}^{(n)} f ) \geq \Einf_{X_{\om,n}}(g_{\om}^{(n)}) \sum_{Z\in \mathring{\cZ}_{\om,f}^{(n)}} \Einf_Z f \geq b_{\om,f}^{(n)}  \Einf_{X_{\om,n}}(g_{\om}^{(n)})   \Einf (f).
\end{align}

\subsection{Variation estimates and Lasota--Yorke inequality} 
For every $\om\in\Om$,  $f\in\BV$, and  $n\in\NN$, we have
\begin{align*}
	\var(\cL_{\om}^{(n)}f)\leq \sum_{Z\in\mathring{\cZ}_{\om}^{(n)}}\var\lt(\ind_{T_\om^{(n)}(Z)} \lt((f g_{\om}^{(n)})\circ T_{\om,Z}^{-n}\rt)\rt).
\end{align*}
For each $Z\in\mathring{\cZ}_{\om}^{(n)}$ we have 
\begin{align}
	&\var\lt(\ind_{T_\om^{(n)}(Z)} \lt((f g_{\om}^{(n)})\circ T_{\om,Z}^{-n}\rt)\rt)
	\leq \var_Z(f g_{\om}^{(n)})+2\Esup_Z\absval{f g_{\om}^{(n)}}
	\nonumber\\
	&\qquad\qquad\leq 3\var_Z(f g_{\om}^{(n)})+2\Einf_Z\absval{f g_{\om}^{(n)}}
	\nonumber\\
	&\qquad\qquad\leq 3\norm{g_{\om}^{(n)}}_{\infty}\var_Z(f)+3\Esup_Z|f|\var_Z(g_{\om}^{(n)})+2\norm{g_{\om}^{(n)}}_{\infty}\Einf_Z|f|.
	\label{eq:var ineq over partition}
\end{align}
An inductive argument starting from the bound $\var(f h)\leq \var(f)\|h\|_\infty + \var(h)\|f\|_\infty$, and considering that $T_\om^{(n)}$ is monotonic on $Z$, yields
\begin{equation*}\label{eq:vargn}
	\var_Z(g_\om^{(n)}) \leq \|g_\om\|_\infty^{(n)} \sum_{j=0}^{n-1} \frac{\var(g_{\sg^j\om})}{\|g_{\sg^j\om}\|_\infty},
\end{equation*}
where $\|g_\om\|_\infty^{(n)}:=\prod_{j=0}^{n-1}\|g_{\sg^j\om}\|_\infty$.
Let $\tilde{S}_{n,\om}(g):=\sum_{j=0}^{n-1} \frac{\var(g_{\sg^j\om})}{\|g_{\sg^j\om}\|_\infty}$.
Therefore, \eqref{eq:var ineq over partition} yields
\begin{align*}
	\var\lt(\ind_{T_\om^{(n)}(Z)} \lt((f g_{\om}^{(n)})\circ T_{\om,Z}^{-n}\rt)\rt)
	&\leq \big(3+3 \tilde{S}_{n,\om}(g)\big)\norm{g_{\om}}^{(n)}_{\infty} \var_Z(f)
	\\
	&+\big(2+3\tilde{S}_{n,\om}(g)\big) \norm{g_{\om}}^{(n)}_{\infty}\Einf_Z|f|.\nonumber
\end{align*}
Thus,
\begin{align*}
	\var({\cL^{(n)}_{\om} f}) &\leq 
	\big(3+3 \tilde{S}_{n,\om}(g)\big)
	\norm{g_{\om}}^{(n)}_{\infty}\var(f)\\
	&+	\big(2+3\tilde{S}_{n,\om}(g)\big) \norm{g_{\om}}^{(n)}_{\infty}
	\lt(\sum_{Z\in\mathring{\cZ}_{\om,f}^{(n)}}\Einf_Z|f|+\sum_{Z\in\mathring{\cZ}_{\om,p}^{(n)}}\Einf_Z|f|\rt).
\end{align*}	
Grouping as in \eqref{eq:sumNF}, one gets
\begin{align*}
	\var(\cL_{\om}^{(n)}f)&\leq 
	\big(3+3 \tilde{S}_{n,\om}(g)\big) (1+2\xi_{\om}^{(n)})
	\norm{g_{\om}}^{(n)}_{\infty}\var(f)\\
	\\
	&+\big(2+3\tilde{S}_{n,\om}(g)\big) (1+2\xi_{\om}^{(n)}) \norm{g_{\om}}^{(n)}_{\infty}
	\sum_{Z\in\cZ_{\om,f}^{(n)}}\Einf_Z|f|.
\end{align*}
Furthermore, if $f\geq0$, \eqref{eq:infLf} implies
\begin{align} \label{eq:LY+}
	\var(\cL_{\om}^{(n)}f)&\leq 
	\big(3+3 \tilde{S}_{n,\om}(g)\big) (1+2\xi_{\om}^{(n)})
	\norm{g_{\om}}^{(n)}_{\infty} \lt( \var(f) +  \frac{\Einf({\cL_{\om}^{(n)}f}) }
	{ \Einf_{X_{\om,n}}(g_{\om}^{(n)}) } \rt).
\end{align}

\section{(Strictly) invariant cones and strongly contracting potentials}\label{sec:invCones}
Given $a>0$, we consider the cones
\[
\cC_{a} =\lt \{ f\in \BV: f > 0, \var(f)\leq a \Einf (f) \rt\}\subset \BV.
\]
This is a positive, convex cone with non-empty interior. Also, $\cC_{a} \cup \{0\}$ is closed. Let $\preceq_{a}$ be the partial order induced by $\cC_{a}$. That is, $f\preceq_{a} g$ iff $f-g \in \cC_{a}\cup\{0\}$.
Then, $(\BV,\preceq_{a})$ is integrally closed\footnote{$(V, \preceq)$ is integrally closed if for every $\al_n \to \al \in \RR$, $f, g \in V$ such that $0\preceq f,g$  and $\al_n f \preceq g$, $\al f \preceq g$.}. In addition, every $f\in \BV$ may be written as $f=f_1-f_2$ such that $f_1, f_2\in \cC_{a}$, for instance, by choosing $f_1=f+c, f_2=c$ for sufficiently large $c>0$.

The inequalities \eqref{eq:infLf} and \eqref{eq:LY+} yield the following. 
\begin{lemma}
	If $f\in \cC_{a}$ and $n\in \NN$, then  
	$\cL_{\om}^{(n)}f \in \cC_{a'}$, with 
	\begin{equation}\label{eq:cndn}
		a'= \big(3+3 \tilde{S}_{n,\om}(g)\big) (1+2\xi_{\om}^{(n)})
		\frac{\norm{g_{\om}}^{(n)}_{\infty}}{\Einf_{X_{\om,n}}(g_{\om}^{(n)})} \lt( \frac{a}{b_{\om,f}^{(n)} } +  1 \rt)
		=: c_{\om,n} a + d_{\omega,n}.
	\end{equation}
\end{lemma}

The  next definition will be key for our arguments, as it allows for the construction of an invariant family of random cones, using ideas going back to Kifer  \cite{kifer_thermodynamic_2008}; see also \cite{stadlbauer_suzuki_varandas_2020}. 

\begin{definition}\label{def:cp}
	We say $\{\log g_\om\}$ is a  \emph{(random) strongly contracting potential} for the random (open or closed) map $\{(T_\om, H_\om)\}$ if
	$\log \#\mathring{\cZ}_{\om}, \log \|g_\om\|_\infty, \log   \Einf(g_{\om}),  \frac{\var(g_{\om})}{\|g_{\om}\|_\infty} \in L^1(m)$ and
	there exists $n_*>0$ such that
	$\int \log c_{\om,n_*} dm<0$, where $c_{\om,n}$ is defined in \eqref{eq:cndn}.
\end{definition}

\begin{remark}
	This condition is related to, but more restrictive than, the definitions of contracting potential in \cite{liverani_conformal_1998} (autonomous setting) and  \cite[Definition 2.15]{AFGTV1}, \cite[(Q1)]{AFGTV2} (random setting). On the other hand, \cite{liverani_conformal_1998, AFGTV1, AFGTV2} also require a covering condition, which is not required in this work.
	In \cite{stadlbauer_suzuki_varandas_2020}, the authors investigate random (closed) 
	non-uniformly expanding $C^1$
	maps with $C^1$ potentials satisfying a contracting-like condition. In Remark~\ref{rmk:ssv}, we show that,  in the one-dimensional setting, this condition is more restrictive than that of Definition~\ref{def:cp}.
\end{remark}
\begin{lemma}\label{lem:contractingCones}
	Assume $\{\log g_\om\}$ is a random  strongly contracting potential for the random (open or closed) map $\{(T_\om, H_\om)\}$.
	Then, there exists $n_*\in \NN$, $0<\gm<1$ and a 
	family of cones $(\cC_{a_\om})_{\om \in \Om}$
	which is invariant under $\cL_{\om}^{(n_*)}$ and satisfies
	$\cL_{\om}^{(n_*)}\cC_{a_\om} \subset \cC_{\gm a_{\sg^{n_*}\om}}$. 
	Furthermore  $a_\om$ may be chosen as in \eqref{eq:aom}, and therefore it may be assumed to be tempered.
\end{lemma}

\begin{proof}
	The hypotheses ensure there exists $n_*\in \NN$ such that $\int \log c_{\om,n_*} dm<0$, where $c_{\om,n_*}$ is defined in \eqref{eq:cndn}. Thus, one can find $0<\gm<1$ such that $\int \log c_{\om,n_*} d\,m=:\log \tilde{\gm}<\log \gm<0$. Then, it follows that the twisted cohomological equation $\gm a_{\sg^{n_*}\om}=c_{\om,n_*} a_{\om} + d_{\om,n_*}$ has a measurable, $m$-almost surely finite solution given by 
	\begin{equation}\label{eq:aom}
		a_\om= \sum_{j=0}^\infty \gm^{-j-1} d_{\sg^{-j-1}\om,n_*} \prod_{k=1}^j c_{\sg^{-k}\om,n_*},
	\end{equation}
	where, for convenience, we let $\Pi_{k=1}^0 c_{\sg^{-k}\om,n_*}:=1$.

	The fact that $a_\om$ is $m$-almost surely finite and tempered is a consequence of the integrability assumptions in Definition~\ref{def:cp}, combined with sub-multiplicativity of $1/\Einf(g_{\om}^{(n)})$. Indeed,
	notice that $b_{\om,f}^{(n)}, \xi_{\om}^{(n)} \leq \prod_{j=0}^{n-1}\#\mathring{\cZ}_{\sg^j\om}$.
	Hence, $d_{\om,n_*}$ is log-integrable, where $d_{\om,n_*}$ is defined in \eqref{eq:cndn}. Hence, there exists $\ep>0$ satisfying  $e^{2\ep}\tilde{\gamma}\le\al\gamma$ for $0<\al<1$ and a tempered measurable function $D_\om$ such that $d_{\sg^{-j-1}\om,n_*}\leq D_\om e^{\ep j}$.  
	Similarly, there is a tempered measurable function $C_\om$ such that $\prod_{k=1}^j c_{\sg^{-k}\om,n_*}\leq C_\om e^{j\ep} \tilde{\gm}^j$.
	Therefore, substituting into \eqref{eq:aom}, we get
	$
	a_\om \leq C_\om D_\om/ \big(\gm-e^{2\ep} \tilde{\gm}\big)
	\leq C_\om D_\om /(\gamma(1-\alpha))
	$
	is tempered.
	It is straightforward to verify that $\cL_{\om}^{(n_*)}\cC_{a_\om} \subset \cC_{\gm a_{\sg^{n_*}\om}}$.
\end{proof}

\subsection{Contraction of projective metric} 
In the setting of Lemma~\ref{lem:contractingCones},
let $\preceq_\om$ be the partial order induced by $\cC_{a_\om}$. That is, $f\preceq_\om g$ iff $f-g \in \cC_{a_\om}\cup\{0\}$.
Let  $\Ta_{\om}$ be the  Hilbert (projective) pseudo metric on $\cC_{a_\om}$, given by 
\begin{align*}
	\Ta_\om(f,h):=\log\frac{\rho_\om(f,h)}{\tau_\om(f,h)},
\end{align*}
where $f,g \in \cC_{a_\om}$, 
$\tau_\om(f,h):=\sup\set{\lm>0: \lm f\preceq_\om h} 
\text{ and }
\rho_\om(f,h):=\inf\set{\mu>0: \mu f\succeq_\om h}$; the distance is infinite if the numerator is $\infty$ or the denominator is 0.
\begin{lemma}	\label{lem:finiteDiam}
	Assume $0<\gm<1$ and $f\in\cC_{\gm a_\om}$. Then,
	\begin{equation}\label{eq:findiam}
		\Ta_\om(f,1)\leq \log \frac{1+\gm (a_\om+1)}{1-\gm}=:\Delta_\om/2.
	\end{equation}
	Thus, the diameter of $\cC_{\gm a_\om}$ as a subset of $\cC_{a_\om}$ is at most $\Dl_\om<\infty$.
\end{lemma}
\begin{proof}
	Let $f\in\cC_{\gm a_\om}$.
	First, $\lm \preceq_\om f$ if and only if $\lm \leq \Einf (f)$ and $\var(f)=\var(f-\lm)\leq a_\om \Einf(f-\lm)$.
	This happens if $\lm\leq (1-\gm)\Einf f$.
	Also, $f \preceq_\om \mu$ if and only if $\|f\|_\infty \leq \mu$ and $\var(f)=\var(\mu-f)\leq a_\om \Einf(\mu-f)$. Since $\var(f)\leq \gm a_\om \Einf(f)$ and $\|f\|_{\infty}\leq (1+\gm a_\om)\Einf(f)$, this happens if $\mu\geq(\gm+1+\gm a_\om)\Einf(f)$. Thus, we conclude that
	$
	\Ta_\om(f,1)\leq \log \frac{1+\gm (a_\om+1)}{1-\gm},
	$
	as claimed.
\end{proof}

\begin{lemma}\label{lem:coneCont}
	Under the hypotheses of Lemma~\ref{lem:contractingCones}, there exists  $0<\vartheta<1$ such that for every $k\geq 0$, and \maeom,
	\begin{equation}\label{eq:coneContraction}
		\Ta_{\om}(\cL_{\sg^{-n_*l}\om}^{n_*l} f,\cL_{\sg^{-n_*(l+k)}\om}^{n_*(l+k)} h) \leq  \Ta_{\sg^{-ln_*}\om}( f,\cL_{\sg^{-kn_*}\om}^{kn_*} h) \vartheta^l ,
	\end{equation}
	for every sufficiently large $l$ (depending on $\om$), every $f\in \cC_{a_{\sg^{-ln_*}\om}}$ and every $h\in \cC_{a_{\sg^{-n_*(l+k)}\om}}$. 
\end{lemma}
\begin{proof}
	Lemma~\ref{lem:contractingCones} implies $\cL_{\sg^{-n_*l}\om}^{(n_*)}\cC_{a_{\sg^{-n_*l}\om}, \sg^{-n_*l}\om} \subset \cC_{\gm a_{\sg^{-n_*(l-1)}\om}}$ and Lemma~\ref{lem:finiteDiam} implies $\diam \lt (\cL_{\sg^{-n_*l}\om}^{(n_*)}\cC_{a_{\sg^{-n_*l}\om}} \rt) \leq \Dl_{\sg^{-n_*(l-1)}\om}$, where
	$\Dl_\om$ is as in \eqref{eq:findiam}. Let
	$\ep>0$ and $D \in \RR$ be such that $m(\{\om\in \Om : \Dl_\om \leq D\})>1-\ep/n_*$. 
	Recall the projective metric is weakly contracted by $\cL_\om^{(n_*)}$ for \maeom, and, once the diameter of the image is finite, it is strictly contracted by a factor of $\tanh(\tfrac{D}{4})$ whenever $\Dl_{\om}<D$. Hence, by ergodicity of $\sg$, \eqref{eq:coneContraction} holds for sufficiently large $l$, provided $\vartheta>(\tanh(\tfrac{D}{4}))^{1-\ep}$.
\end{proof}

\begin{remark}
	For simplicity and clarity of presentation, we assume from now on that $$n_*=1.$$  \cite{AFGTV1,AFGTV2} address the possibility of $n_*>1$ in a related setting. 
\end{remark}

\section{Construction of equivariant densities and conformal measures}\label{S:constructions}

In this section, we construct equivariant densities and conformal measures for the random map $\{(T_\om, H_\om)\}$ with
strongly contracting potential $\{\log g_\om\}$.
We point out that these constructions are completely decoupled, in contrast to the standard approach of establishing the existence of conformal measures first, and using them to build the densities. 
(See Remark~\ref{rmk:Comparison} for further details on this comparison.)

Note that the norm $\|f\|_\infty$ is compatible with $\preceq_\om$. That is, for all $f,h\in BV$, if $-f\preceq_\om h \preceq_\om f$ then $\|h\|_{\infty}\leq \|f\|_{\infty}$.
Also, the function $\Einf: \cC_{a_\om} \to \RR_+$ is homogeneous and $\preceq_\om$ preserving.
Hence, as in \cite[Lemma 2.2]{liverani_conformal_1998}, 
for every $f, h \in \cC_{a_\om}$ such that $\Einf f=\Einf h>0$, we have
\begin{equation}\label{eq:compDist}
	\| f-h\|_{\infty}\leq (e^{ \Ta_{\om}(f,h)}-1) \min (\|f\|_{\infty}, \|h\|_{\infty}).
\end{equation}

\subsection{Equivariant densities}\label{S:equivDens} 
In this section, we show the following.
\begin{lemma}\label{lem:eqdens}
	Assume $\{\log g_\om\}$ is a  strongly contracting potential for the random (open or closed) map $\{(T_\om, H_\om)\}$, and $a_\om$ is as in \eqref{eq:aom}. Then,
	\begin{enumerate}[(i)]
		\item \label{it:eqdens1}
		For each $f\in \cC_{1}$ the sequence $\frac{{\nothing}\cL_{\sg^{-n}\om}^{(n)} f}{ \Einf(\cL_{\sg^{-n}\om}^{(n)} f)}$ is Cauchy with respect to $\|\cdot\|_\infty$. 
		Hence, the following limit exists:
		$
		q_\om^f:= \lim_{n\to\infty}\frac{{\nothing}\cL_{\sg^{-n}\om}^{(n)} f}{ \Einf(\cL_{\sg^{-n}\om}^{(n)} f)}$.
		Furthermore, 
		$ \Einf  (q_\om^f)=1  {\text{ and }}  \var(q_\om^f)\leq \gm a_\om.$
		In addition,
		$\cL_\om q_\om^f=\lm_\om^f q_{\sg\om}^f$,
		with $\lm_\om^f=\Einf(\cL_\om q_\om^f)$.
		\item \label{it:eqdens2}
		The functions $q_\om^f$ and multipliers $\lm_\om^f$ are independent of $f$. Call them $q_\om$ and $\lm^-_\om$, respectively. 
		Then, $\cL_\om q_\om = \lm^-_\om q_{\sg\om}$,
		\begin{equation}\label{eq:equiv-}
			\begin{split}
				q_\om= {\nothing}\lim_{n\to\infty}\frac{\cL_{\sg^{-n}\om}^{(n)} 1}{ \Einf(\cL_{\sg^{-n}\om}^{(n)} 1)},\qquad
				\lm^-_\om= \lim_{n\to\infty}\frac{\Einf(\cL_{\sg^{-n}\om}^{(n+1)} 1)}{ \Einf(\cL_{\sg^{-n}\om}^{(n)} 1)}=\Einf (\cL_\om q_\om).
			\end{split}
		\end{equation}
	\end{enumerate}
\end{lemma}
\begin{proof}
	To show \ref{it:eqdens1},
	first note that \eqref{eq:aom} implies $a_\om\geq1$. Thus,  $\cC_{1}\subset \cC_{a_\om}$ for \maeom. 
	Let $f_n:=\frac{{\nothing}\cL_{\sg^{-n}\om}^{(n)} f}{ \Einf(\cL_{\sg^{-n}\om}^{(n)} f)}$. Using \eqref{eq:compDist}, we have, for $m>n\geq1$,
	\begin{equation}\label{eq:Cauchyfn}
		\|f_n -f_m\|_\infty \leq (e^{\Ta_{\om}(f_n,f_m)}-1) \|f_n\|_\infty.
	\end{equation}
	Since $f_n \in \cC_{\gm a_\om}$ and $\Einf f_n=1$, then $\|f_n\|_\infty\leq 1+ \gm a_\om$.
	On the other hand, by \eqref{eq:coneContraction}, for sufficiently large $n$, $\Ta_{\om}(f_n,f_m)\leq \Dl_{\sg^{-n+1}\om} \vartheta^{n}$, where $\Dl_\om$ is as in \eqref{eq:findiam} and $\vartheta<1$ is as in Lemma~\ref{lem:coneCont}. Since $a_\om$ is tempered, so is $\Dl_\om$, and \eqref{eq:Cauchyfn} tends to 0 exponentially as $n\to \infty$.
	Hence, the following limit exists in $L^\infty$:
	\begin{align*}
		q_\om^f&:= \lim_{n\to\infty}\frac{{\nothing}\cL_{\sg^{-n}\om}^{(n)} f}{ \Einf(\cL_{\sg^{-n}\om}^{(n)} f)}.
	\end{align*}
	Also,
	$\Einf  (q_\om^f)=1$ and $\var(q_\om^f)\leq \limsup \var(f_n)\leq \gm a_\om$.
	In addition,
	\begin{equation}\label{eq:equivDens}
		\begin{split}
			{\nothing}\cL_\om q_\om^f&=  \lim_{n\to\infty}\frac{{\nothing} \cL_{\sg^{-n}\om}^{(n+1)} f}{\Einf(\cL_{\sg^{-n}\om}^{(n+1)} f)}\frac{\Einf(\cL_{\sg^{-n}\om}^{(n+1)} f)}{ \Einf(\cL_{\sg^{-n}\om}^{(n)} f)}=  q_{\sg\om}^f \lim_{n\to\infty}\frac{\Einf(\cL_{\sg^{-n}\om}^{(n+1)} f)}{ \Einf(\cL_{\sg^{-n}\om}^{(n)} f)}=:\lm_\om^f q_{\sg\om}^f.
		\end{split}
	\end{equation}
	The normalization of $q_\om^f$ implies that $\lm_\om^f=\Einf(\cL_\om q_\om^f)$.

	To show \ref{it:eqdens2}, we show there exists $q_\om\in BV$ such that
	$q_\om=q_\om^f$ for every $f\in \cC_{1}$. 
	Indeed, for $f, h\in \cC_{1}$ we have, for every $n\in\NN$,
	\begin{equation}\label{eq:qomUnique}
		\|q_\om^f-q_\om^h\|_\infty \leq (e^{\Ta_{\om}(q_\om^f,q_\om^h)}-1) \|q_\om^f\|_\infty
		\leq (e^{\Ta_{\om}(\cL_{\sg^{-n}\om}^{(n)} q_{\sg^{-n}\om}^f,\cL_{\sg^{-n}\om}^{(n)} q_{\sg^{-n}\om}^h)}-1) \|q_\om^f\|_\infty.
	\end{equation}
	By \eqref{eq:coneContraction}, $\Ta_{\om}(\cL_{\sg^{-n}\om}^{(n)} q_{\sg^{-n}\om}^f,\cL_{\sg^{-n}\om}^{(n)} q_{\sg^{-n}\om}^h)\leq \Dl_{\sg^{-n+1}\om} \vartheta^{n-1}$ for sufficiently large $n$, and $\|q_\om^f\|_\infty\leq 1+\gm a_\om$. Using once again that $\Dl_\om$ is tempered, we conclude that the RHS of \eqref{eq:qomUnique} tends exponentially fast to 0 as $n\to\infty$. Thus, $q_\om^f=q_\om^h=:q_\om$.
	Hence, \eqref{eq:equivDens} implies that $\lm_\om^f$ is also independent of $f$, call it $\lm^-_\om$. Thus, \eqref{eq:equiv-} holds.
\end{proof}

\subsection{Equivariant conformal measures}\label{S:confMeas}
In this section, we show the following.
\begin{lemma}\label{lem:confmeas}
	Assume $\{\log g_\om\}$ is a   strongly contracting potential for the random (open or closed) map $\{(T_\om, H_\om)\}$. Then,
	for each $f\in \cC_{1}$, the sequence $\frac{\Einf(\cL_{\om}^{(n)} f)}{\Einf(\cL_{\om}^{(n)} 1)}$ is Cauchy.
	Its limit,
	\begin{equation}\label{eq:nu}
		\nu_\om(f):=\lim_{n\to\infty}\frac{\Einf(\cL_{\om}^{(n)} f)}{\Einf(\cL_{\om}^{(n)} 1)},
	\end{equation}
	defines a positive linear functional which can be extended by linearity to $BV$, and to a non-atomic probability measure with support contained in $X_\om$. 
	Furthermore, 
	$\nu_\om$ satisfies $\nu_{\sg\om}(\cL_\om f)=\lm^+_\om \nu_\om(f)$, where
	\begin{equation}\label{eq:confmeas}
		\lm^+_\om=\lim_{n\to\infty}\frac{\Einf_{\nothing} (\cL_{\om}^{(n+1)} 1)}{\Einf_{\nothing}(\cL_{\sg\om}^{(n)} 1)}=  \nu_{\sg\om}(\cL_\om 1).
	\end{equation}
\end{lemma}
\begin{proof}
	To show the sequence is Cauchy, it suffices to show that there exists $C_\om>0$  such that for every $f\in\cC_{1}$,
	\begin{equation}\label{eq:cmcauchy}
		\lm_{\sg^n\om}^{-,(k)} (1 -\ C_\om r^n) \leq \frac{\Einf_{\nothing} (\cL_{\om}^{(n+k)} f)}{\Einf_{\nothing}(\cL_{\om}^{(n)} f)} \leq \lm_{\sg^n\om}^{-,(k)} (1 +\ C_\om r^n),
	\end{equation}
	where $\lm_{\om}^{-,(k)}:=\lm_\om^- \lm_{\sg\om}^-\cdots \lm_{\sg^{k-1}\om}^-$, $\lm_\om^->0$ is as in \eqref{eq:equiv-} and $\vartheta<r<1$, with $\vartheta$ is as in Lemma~\ref{lem:coneCont}.
	To see this, we argue as in \S\ref{S:equivDens}. Thus, there exists $C_\om>0$ such that 
	$\lt\| \frac{\nothing \cL_{\om}^{(n)} f}{\Einf(\cL_{\om}^{(n)} f)}- q_{\sg^n\om}\rt\|_\infty <\ C_\om r^n$. Hence,
	\begin{align*}
		\frac{\Einf_{\nothing} (\cL_{\om}^{(n+k)} f)}{\Einf(\cL_{\om}^{(n)} f)}&=
		\Einf_{\nothing} \cL_{\sg^n\om}^{(k)}
		\lt( \frac{\nothing \cL_{\om}^{(n)} f}{\Einf(\cL_{\om}^{(n)} f)}\rt) \\
		&\leq 
		\Einf_{\nothing} \lt( \cL_{\sg^n\om}^{(k)}
		\lt( q_{\sg^n\om} (1+ \ C_\om r^n) \rt) \rt) =\lm_{\sg^n\om}^{-,(k)} (1+ \ C_\om r^n),  
	\end{align*}
	where in the next to last step we have used that $\Einf_{\nothing} q_{\sg^n\om}=1$.
	The lower bound in \eqref{eq:cmcauchy} is obtained similarly.
	
	Let 
	$\nu_\om(f):=\lim_{n\to\infty}\frac{\Einf(\cL_{\om}^{(n)} f)}{\Einf(\cL_{\om}^{(n)} 1)}$.
	Positivity and linearity of $\nu_\om$ are clear. Since $\cC_{1}$ has non-empty interior, $\nu_\om$ can be extended by linearity to $BV$.
	Since $|\nu_\om(f)|\leq \|f\|_\infty$ and $\nu_\om(1)=1$, by the Riesz representation theorem, $\nu_\om$ gives rise to a probability measure $\tilde \nu_\om$, with  $\supp(\tilde \nu_\om)\subseteq X_\om$. 
	
	If $Z\in \mathring{\cZ}_{\om}^{(k)}$, then $T_\om^{(k)}|_Z: Z \to \IO$ is injective, so
	$\|\cL_{\om}^{(k)} \ind_Z \|_\infty \leq \|g_\om^{(k)}\|_{\infty}$. Thus, 
	\begin{align*}
		\nu_\om(\ind_Z) &=
		\lim_{n\to\infty}\frac{\Einf_{\nothing} (\cL_{\sg^k\om}^{(n-k)} \cL_{\om}^{(k)} \ind_Z)}{\Einf(\cL_{\om}^{(n)} 1)}
		\\
		&\leq \|g_\om^{(k)}\|_\infty \lim_{n\to\infty}\frac{\Einf_{\nothing} (\cL_{\sg^k\om}^{(n-k)} 1)}{\Einf(\cL_{\sigma^k\om}^{(n-k)} 1)\Einf(\cL_{\om}^{(k)} 1)}\leq
		\frac{\|g_\om^{(k)}\|_\infty}{b_{\om,f}^{(k)} \Einf_{X_{\om,k}}(g_\om^{(k)})}.
	\end{align*}
	Since $\{\log g_\om\}$ is strongly contracting, Kingman's subadditive ergodic theorem ensures the upper bound approaches 0 as $k\to\infty$. 
	Thus, $\lim_{k\to \infty}\max_{Z\in \mathring{\cZ}_{\om}^{(k)}} \nu_\om(\ind_Z)=0$.
	Hence, $\nu_\om$ is non-atomic, and standard approximation arguments ensure  that for every $J\subset I$, 
	$\nu_\om(\ind_J)=\tilde \nu_\om (J)$, so we also write $\nu_\om$ to refer to the measure $\tilde \nu_\om$.

	For the final claim, we have
	\begin{equation*}
		\begin{split}
			\nu_{\sg\om}(\cL_\om f)&=\lim_{n\to\infty}\frac{\Einf_{\nothing} (\cL_{\om}^{(n+1)} f)}{\Einf_{\nothing}(\cL_{\sg\om}^{(n)} 1)} \\
			&= \lim_{n\to\infty}\frac{\Einf_{\nothing} (\cL_{\om}^{(n+1)} f)}{\Einf_{\nothing}(\cL_{\om}^{(n+1)} 1)}
			\frac{\Einf_{\nothing} (\cL_{\om}^{(n+1)} 1)}{\Einf_{\nothing}(\cL_{\sg\om}^{(n)} 1)}= \nu_\om(f) \nu_{\sg\om}(\cL_\om 1).
		\end{split}
	\end{equation*}
\end{proof}

\begin{remark}\label{rmk:Comparison}
	The construction of conformal measures here may be regarded as a random version of that in \cite{liverani_conformal_1998}. On the other hand, the densities constructed in \cite{liverani_conformal_1998} differ from ours in the normalisation. If we denote their densities by $\tilde{q}_\om$, they are normalised so that $\nu_\om(\tilde{q}_\om)=1$. As it can be deduced from the upcoming \eqref{eq:lamnu}, this choice ensures that their corresponding multipliers, $\tilde{\lm}_\om$, satisfy $\tilde{\lm}_\om=\lm_\om^+$. 
\end{remark}

\section{Main results}
\subsection{Equilibrium states and exponential decay of correlations}\label{S:pfMainThm}
In this section we show the following.
\begin{theorem}\label{thm:qinvMeas}
	Assume  $\{\log g_\om=:\phi_\om\}$ is a   strongly contracting potential for the random (open or closed) map $\{(T_\om, H_\om)\}$.
	Let $\lm^\pm_\om$,
	$q_\om$ and  $\nu_\om$ and  be as in \S\ref{S:equivDens} and \S\ref{S:confMeas}. 
	Then\footnote{We will show in Theorem~\ref{thm:randompf} that  $\Lm_1=\lim_{n\to \infty}\frac1n \log\|\cL_\om^{(n)}\|_{\BV}$ for \maeom.}, 
	$\int \log \lm_\om^+dm= \int \log \lm_\om^-dm=:\Lm_1$.
	Define the probability measures $\mu_\om$ by $\int f d\mu_\om := \frac{\int f q_\om d\nu_\om}{\nu_\om(q_\om)}$. 
	Then,
	\begin{equation}\label{eq:invmeas}
		\int f   d\mu_{\sg \om} =  \int f \circ T_\om d\mu_\om.
	\end{equation}
	
	Furthermore, there exist a tempered $C_\om>0$ and  $0<r<1$ such that for every $f\in L^1(\nu_\om)$, $\~f\in L^1(\nu_{\sg^n\om})$, and $h \in BV$ 
	\begin{equation}
		\label{eq:docPast}
		\big| \mu_{\sg^{-n}\om}(f\circ  T^{(n)}_{\sg^{-n}\om} \cdot h)  - \mu_{\om}(f)\mu_{\sg^{-n}\om}(h)\big|
		\leq C_\om\|f\|_{L^1(\nu_{\om})} \|h\|_{BV} r^n, \quad \text{and}
	\end{equation}
	\begin{equation}
		\label{eq:docFuture}
		\big| \mu_\om(\~f\circ  T^{(n)}_\om \cdot h)  - \mu_{\sg^n\om}(\~f) \mu_{\om}(h)\big|
		\leq C_\om\|\~f\|_{L^1(\nu_{\sg^n\om})} \|h\|_{BV} r^n.
	\end{equation}
	In fact, \eqref{eq:docPast} and \eqref{eq:docFuture} hold for any choice $r>\vartheta$, with $\vartheta$ as in Lemma~\ref{lem:coneCont}.
\end{theorem}

\begin{remark}
	The quantity $\Lm_1$ in Theorem~\ref{thm:qinvMeas} is called the \emph{maximal Lyapunov exponent} of the cocycle generated by $\{\cL_\om\}$  in the context of multiplicative ergodic theory; and the \emph{expected pressure}, denoted by $\cEP(\phi)$, in the thermodynamic formalism approach. The proof of Theorem~\ref{thm:randompf} will show that the \emph{second Lyapunov exponent} of the cocycle  satisfies $\lm_2\leq \tanh(\tfrac{D}{4})$, with the notation of Lemma~\ref{lem:coneCont}. This bound is related to the upper bound of \cite{Horan1}.
\end{remark}

We extend the notion of invariant measures corresponding to punctured potentials introduced in \cite{DemersTodd-Multimodal17}, to the random setting. Let $\cP_{T,m}^H(I)$ denote the collection of $T$-invariant probability measures $\eta$ on $\Om\times I$ with marginal $m$ on $\Omega$, 
such that its disintegration $\{\eta_\om\}$ satisfies $\eta_\om(H_\om)=0$ for $m$ a.e. $\om\in\Om$. 

\begin{definition}
	We say that a measure $\eta\in\cP_{T,m}^H(\IO)$ is a \emph{relative equilibrium state} for the
	random map $\{(T_\om, H_\om)\}$ with
	potential $\{\phi_\om\}$ if $$\cEP(\phi)=h_\eta(T)+\int_{\Om \times \IO} \phi \,d\eta,$$
	where $h_\eta(T)$ denotes the entropy of $T$ with respect to $\eta$.
\end{definition}

The proof of the next result follows similarly to the proof of Theorem 2.23 in \cite{AFGTV1} (see also Remark 2.24, Lemma 12.2 and Lemma 12.3).
\begin{theorem}\label{thm:eqstates}
	Assume  $\{\log g_\om=:\phi_\om\}$ is a strongly contracting potential for the random (open or closed) map $\{(T_\om, H_\om)\}$. Then, the random measure $\mu\in\cP_{T,m}^H(\IO)$ with disintegration $\{\mu_\om\}$ produced in  Theorem~\ref{thm:qinvMeas} is the unique relative equilibrium state for $\{\phi_\om\}$. It satisfies the following variational principle:
	\begin{align*}
		\Lm_1= \cEP(\phi)
		= h_\mu(T)+\int_{\Om \times \IO} \phi \,d\mu
		=
		\sup_{\eta\in\cP_{T,m}^H(\IO)} h_\eta(T)+\int_{\Om \times \IO} \phi \,d\eta.
	\end{align*}
\end{theorem}
\begin{remark}
	The same conclusions hold for the random invariant measures $\{\mu_\om\}$ in the random open setting of \cite{AFGTV2}.
\end{remark}

\begin{proof}[Proof of Theorem~\ref{thm:qinvMeas}]
	To show $\int \log \lm_\om^+dm= \int \log \lm_\om^-dm$, we prove that for \maeom,
	\begin{equation}\label{eq:lamnu}
		\frac{ \nu_\om(q_\om)\lm_\om^+ }{ \nu_{\sg\om}(q_{\sg\om})\lm_\om^-}=1.
	\end{equation}
	Indeed,
	\begin{align*}
		\nu_\om(q_\om) &= \lim_{n\to\infty}\frac{\Einf_{\nothing} (\cL_{\om}^{(n)} q_\om)}{\Einf_{\nothing}(\cL_{\om}^{(n)} 1)}=
		\lim_{n\to\infty}\frac{\lm_\om^- \Einf_{\nothing} (\cL_{\sg\om}^{(n-1)} q_{\sg\om})}{\Einf_{\nothing}(\cL_{\sg\om}^{(n-1)} (\cL_\om 1))}=\frac{\lm_\om^- \nu_{\sg\om}(q_{\sg\om})}{\lm_\om^+}.
	\end{align*}
	Next we show \eqref{eq:invmeas}. In view of Lemma~\ref{lem:eqdens} and Lemma~\ref{lem:confmeas},
	\begin{align*}
		\int f   d\mu_{\sg\om} &= \frac{1}{ \nu_{\sg\om}(q_{\sg\om})}\int f \cdot q_{\sg \om} d\nu_{\sg\om} =  \frac{1}{ \nu_{\sg\om}(q_{\sg\om})\lm_\om^-} \int f \ \cdot \cL_\om (q_\om) d\nu_{\sg\om}\\
		&=  \frac{1}{ \nu_{\sg\om}(q_{\sg\om})\lm_\om^-} \int  \cL_\om (f\circ T_\om \cdot q_\om) d\nu_{\sg\om}
		= \frac{\lm_\om^+}{ \nu_{\sg\om}(q_{\sg\om})\lm_\om^-} \int f\circ T_\om \cdot q_\om d\nu_\om \\
		&= \frac{ \nu_\om(q_\om)\lm_\om^+ }{ \nu_{\sg\om}(q_{\sg\om})\lm_\om^-}  \int f\circ T_\om d\mu_\om.
	\end{align*}
	Then, \eqref{eq:invmeas} follows from \eqref{eq:lamnu}.
	
	For the second part of the theorem, notice that for every $h\in BV$, $(h+c_{h}) q_{\sg^{-n}\om}\in \cC_{\sqrt{\gm} a_{\sg^{-n}\om}}$ for
	$
	c_{h} = \frac{1 + 2\sqrt{\gm}}{\sqrt{\gm}-\gm}\|h\|_{BV}.
	$\footnote{We do not claim this choice of $c_h$ is optimal.}
	This follows from basic properties of variation, and the facts that $a_{\sg^{-n}\om}\geq1, q_{\sg^{-n}\om} \in \cC_{\gm a_{\sg^{-n}\om}}$.
	Furthermore, the invariance property \eqref{eq:invmeas} implies that the left hand side of \eqref{eq:docPast}
	is unchanged if $h$ is replaced by $h+c$ for any $c\in \RR$.
	In the case $c=c_{h}$, the corresponding right hand side changes in that $\|h\|_{BV}$ must be replaced by
	$\|h\|_{BV}+ c_{h} 
	\leq \big(1+\frac{1 + 2\sqrt{\gm}}{\sqrt{\gm}-\gm}\big)\|h\|_{BV}$.
	Thus, to show \eqref{eq:docPast} we will assume, without loss of generality\footnote{However, we should keep this assumption in mind  at the end of the proof, where apparently only $\|h\|_\infty$ is relevant, and not $\|h\|_{\BV}$.}, that  $h q_{\sg^{-n}\om}\in \cC_{\sqrt{\gm} a_{\sg^{-n}\om}}$.
	
	Using Lemma~\ref{lem:confmeas} repeatedly, and \eqref{eq:lamnu} in the last step yields
	\begin{equation}\label{eq:doc1}
		\begin{split}
			\mu_{\sg^{-n}\om}&(f\circ  T^{(n)}_{\sg^{-n}\om} \cdot h) = \frac{1}{\nu_{\sg^{-n}\om}(q_{\sg^{-n}\om})} \int f\circ  T^{(n)}_{\sg^{-n}\om} \cdot h q_{\sg^{-n}\om} d\nu_{\sg^{-n}\om}\\
			&= \frac{1}{\nu_{\sg^{-n}\om}(q_{\sg^{-n}\om}) \lm^{+,(n)}_{\sg^{-n}\om}} \int \cL_{\sg^{-n}\om}^{(n)} (f\circ  T^{(n)}_{\sg^{-n}\om} \cdot h q_{\sg^{-n}\om} )d\nu_{\om}\\
			&= \frac{1}{\nu_{\sg^{-n}\om}(q_{\sg^{-n}\om}) \lm^{+,(n)}_{\sg^{-n}\om}} \int f \cdot \cL_{\sg^{-n}\om}^{(n)} ( h q_{\sg^{-n}\om} )d\nu_{\om}\\
			&= \frac{1}{\lm^{-,(n)}_{\sg^{-n}\om}\nu_{\om}(q_{\om})} \int f \cdot \cL_{\sg^{-n}\om}^{(n)} ( h q_{\sg^{-n}\om} )d\nu_{\om}.
		\end{split}
	\end{equation}
	On the other hand,
	\begin{equation}\label{eq:doc2}
		\begin{split}
			\mu_{\sg^{-n}\om}(h)&=\frac{\nu_{\sg^{-n}\om}(h q_{\sg^{-n}\om})}{\nu_{\sg^{-n}\om}(q_{\sg^{-n}\om})}=
			\lim_{k\to\infty}\frac{\Einf_{\nothing}(\cL_{\sg^{-n}\om}^{(n+k)}( h q_{\sg^{-n}\om}))}{\Einf_{\nothing}(\cL_{\sg^{-n}\om}^{(n+k)}(q_{\sg^{-n}\om}))}\\
			&=\frac{\nu_{\om}(\cL_{\sg^{-n}\om}^{(n)}(h q_{\sg^{-n}\om}))}{\nu_{\om}(\cL_{\sg^{-n}\om}^{(n)}(q_{\sg^{-n}\om}))}
			=\frac{\nu_{\om}(\cL_{\sg^{-n}\om}^{(n)}(h q_{\sg^{-n}\om}))}{\lm^{-,(n)}_{\sg^{-n}\om}\nu_{\om}(q_{\om})}.
		\end{split}
	\end{equation}
	Combining  \eqref{eq:doc1} and \eqref{eq:doc2}, we get
	\begin{equation}\label{eq:doc3}
		\begin{split}
			|\mu&_{\sg^{-n}\om}(f\circ  T^{(n)}_{\sg^{-n}\om} \cdot h)  - \mu_{\om}(f)\mu_{\sg^{-n}\om}(h)| \\
			&= 
			\frac{\lt|\nu_\om \big(f \cdot \cL_{\sg^{-n}\om}^{(n)} ( h q_{\sg^{-n}\om}) - \mu_\om(f)\cL_{\sg^{-n}\om}^{(n)}(h q_{\sg^{-n}\om}) \big)\rt|}{\lm^{-,(n)}_{\sg^{-n}\om}\nu_{\om}(q_{\om})}
			= \frac{\lt|\nu_\om \big(\cL_{\sg^{-n}\om}^{(n)} ( h q_{\sg^{-n}\om})(f -\mu_\om(f))\big)\rt|}{\lm^{-,(n)}_{\sg^{-n}\om}\nu_{\om}(q_{\om})}\\
			&\leq \frac{ \Einf(\cL_{\sg^{-n}\om}^{(n)}(h q_{\sg^{-n}\om})) \lt|\nu_\om \big(q_\om (f -\mu_\om(f))\big)\rt|}{\lm^{-,(n)}_{\sg^{-n}\om}\nu_{\om}(q_{\om})}\\
			&\quad +
			\frac{3 \|q_\om\|_\infty \big \| {\nothing}\cL_{\sg^{-n}\om}^{(n)}(h q_{\sg^{-n}\om})- \Einf(\cL_{\sg^{-n}\om}^{(n)}(h q_{\sg^{-n}\om})) q_{\om} \big\|_\infty \|f\|_{L^1(\nu_{\om})} } {\lm^{-,(n)}_{\sg^{-n}\om}\nu_{\om}(q_{\om})},
		\end{split}
	\end{equation}
	where we have used that 
	$|\nu_\om (f -\mu_\om(f))|\leq 3 \|q_\om\|_\infty \|f\|_{L^1(\nu_{\om})}$ in the last line.
	Since $\nu_\om \big(q_\om (f -\mu_\om(f))\big)=0$, it only remains to bound the last term.
	Lemmas~\ref{lem:finiteDiam} and ~\ref{lem:coneCont}, as well as the fact  that
	$q_{\sg^{-n}\om}, h q_{\sg^{-n}\om}\in \cC_{\sqrt{\gm} a_{\sg^{-n}\om}}$,
	show that for sufficiently large $n\in \NN$,
	\begin{align*}
		\Ta_{\om}(\cL_{\sg^{-n}\om}^{(n)}( h q_{\sg^{-n}\om}),\cL_{\sg^{-n}\om}^{(n)}(q_{\sg^{-n}\om}) ) \leq  \Ta_{\sg^{-n}\om}(h q_{\sg^{-n}\om},  q_{\sg^{-n}\om}) \vartheta^n
		\leq 2\log\left( \frac{1+\sqrt{\gm} (a_{\sg^{-n}\om}+1)}{1-\sqrt{\gm}} \right)  \vartheta^n,
	\end{align*}
	where \eqref{eq:findiam} has been used in the final step. Combining with \eqref{eq:compDist} 
	yields
	\begin{align*}
		\Big\| 
		&{\nothing}\cL_{\sg^{-n}\om}^{(n)}(h q_{\sg^{-n}\om})- \Einf(\cL_{\sg^{-n}\om}^{(n)}(h q_{\sg^{-n}\om})) q_{\om} \Big\|_\infty \\
		&\leq (e^{\Ta_{\sg^{-n}\om}(h q_{\sg^{-n}\om},  q_{\sg^{-n}\om}) \vartheta^n}-1 \big)
		\Einf(\cL_{\sg^{-n}\om}^{(n)}(h q_{\sg^{-n}\om}))
		\|q_{\om}\|_{\infty}.
	\end{align*}
	Hence, using the elementary estimate $|e^x-1|\leq 3x$ for $0\leq x \leq 1$, \eqref{eq:doc3} implies that for sufficiently large $n$,
	\begin{equation}\label{eq:doc4}
		\begin{split}
			|\mu_{\sg^{-n}\om}&(f\circ  T^{(n)}_{\sg^{-n}\om} \cdot h)  - \mu_{\om}(f)\mu_{\sg^{-n}\om}(h)| \\
			&\leq 
			\frac{9\log\big( \frac{1+2\sqrt{\gm} a_{\sg^{-n}\om}}{1-\sqrt{\gm}} \big)  \vartheta^n   \Einf(\cL_{\sg^{-n}\om}^{(n)}(h q_{\sg^{-n}\om}))  \|q_{\om}\|^2_{\infty}
				\|f\|_{L^1(\nu_{\om})} } {\lm^{-,(n)}_{\sg^{-n}\om}\nu_{\om}(q_{\om})}\\
			&\leq 9\log\lt( \frac{1+2\sqrt{\gm} a_{\sg^{-n}\om}}{1-\sqrt{\gm}} \rt)   \frac{\|q_{\om}\|^2_{\infty}} {\nu_{\om}(q_{\om})}
			\|f\|_{L^1(\nu_{\om})}  \|h\|_{\infty}  \vartheta^n=: C'_{\sg^{-n}\om} 
			\frac{\|q_{\om}\|^2_{\infty}} {\nu_{\om}(q_{\om})} \|f\|_{L^1(\nu_{\om})}  \|h\|_{\infty}  \vartheta^n,
		\end{split}
	\end{equation}
	where in the last inequality we have used the fact that 
	\[ \Einf(\cL_{\sg^{-n}\om}^{(n)}(h q_{\sg^{-n}\om}))\leq 
	\|h\|_{\infty} \Einf \cL_{\sg^{-n}\om}^{(n)}(q_{\sg^{-n}\om})=\|h\|_{\infty} \lm^{-,(n)}_{\sg^{-n}\om}.\]
	Since $a_\om$ is tempered, $C'_\om$ is tempered and since $\nu_\om(q_\om)\geq 1$ and by Lemma~\ref{lem:eqdens},
	$\|q_\om\|_\infty \leq 1+ \gm a_\om$, \eqref{eq:docPast} holds for any $r>\vartheta$, with $\vartheta$ as in Lemma~\ref{lem:coneCont}, and some tempered $C_\om$.
	
	The proof of \eqref{eq:docFuture} follows from replacing $\om$ with $\sg^n\om$ in \eqref{eq:doc4}, and using temperedness.
\end{proof}

\subsection{Multiplicative ergodic theory and random Ruelle--Perron--Frobenius decomposition}\label{S:met}
Under mild extra assumptions, the
multiplicative ergodic theorem of \cite{FLQ2} applies to cocycles of random maps with strongly contracting potentials, providing uniqueness of the measures $\mu_\om$ from Theorem~\ref{thm:qinvMeas} and further information.
\begin{theorem}\label{thm:randompf} 
	Assume  $\{\log g_\om\}$ is a strongly contracting potential for the random (open or closed) map $\{(T_\om, H_\om)\}$.
	In addition, suppose $\Om$ is a Borel subset of a separable complete metric space, $m$ is a Borel probability measure and $\sg$ is a homeomorphism.
	Then, there is a unique, measurable random Ruelle--Perron--Frobenius type decomposition for the cocycle generated by $\{\cL_\om\}$. That is,
	for \maeom,  there exists a unique (measurable) tuple
	$(\psi_\om, \nu_\om, \lm_\om)$ with
	$\psi_{\om}\in\BV$, $\nu_{\om}\in\BV^*$, the dual space of $\BV$, and $\lm_{\om}\in\CC\bs\set{0}$ 
	such that
	\begin{equation}\label{eq:rpf1}
		\nu_\om(1)=1, \quad \cL_{\om}(\psi_{\om})=\lm_{\om}\psi_{\sg\om},
		\quad\text{ and }\quad
		\nu_{\sg\om}(\cL_{\om}(f))=\lm_{\om}\nu_{\om}(f),
	\end{equation}
	for all $f\in\BV$, which  also satisfies the following: Let
	$Q_{\om}:\BV\to\BV$ be defined by 	$\lm_{\om}^{-1}\cL_{\om}(f)=\nu_{\om}(f) \psi_{\sg\om}+Q_{\om}(f)$. Then,
	\begin{equation}\label{eq:rpf2}
		Q_{\om}(\psi_{\om})=0, \quad \lim_{n\to\infty} \frac{1}{n}\log \|Q_\om^{(n)}\|_{\BV}<0 \quad  \text{ and } \quad
		\nu_{\sg\om}(Q_{\om}(f))=0,
	\end{equation}
	for all $f\in\BV$, where $Q_\om^{(n)}:=Q_{\sg^{n-1}\om} \circ \dots \circ Q_{\sg\om}\circ Q_\om$.
	
	Furthermore, 			
	\begin{equation}\label{eq:lexp}
		\Lm_1=\int \log\lm_\om dm =
		\lim_{n\to\infty} \frac{1}{n}\log \Einf(\cL_{\om}^{(n)}1)
		= \lim_{n\to \infty}\frac1n \log\|\cL_\om^{(n)}\|_{\BV}, \text{ for \maeom}.
	\end{equation}

\end{theorem}
\begin{proof}[Proof of Theorem~\ref{thm:randompf}]
	
	Let $\om \in \Om$.
	Connecting with the notation of \S\ref{S:constructions}, let
	$\lm_\om=\lm_\om^+$ and $\psi_\om=q_\om/\nu_\om(q_\om)$.  Then,
	the only condition in \eqref{eq:rpf1} and \eqref{eq:rpf2} that is not straightforward to derive from Lemma~\ref{lem:eqdens} and Lemma~\ref{lem:confmeas}  is 
	$\lim_{n\to\infty} \frac{1}{n}\log \|Q_\om^{(n)}\|_{\BV}<0$.
To show this,
we first observe, by induction, that 
\begin{equation}\label{eq:Qn}
	Q_\om^{(n)}(f)=(\lm_\om^{(n)})^{-1} \cL_\om^{(n)}  (f-\nu_\om(f)\psi_\om)=(\lm_\om^{(n)})^{-1} \cL_\om^{(n)}  (f)-\nu_\om(f) \psi_{\sg^n\om}.
\end{equation}
Next, using the notation of Lemma~\ref{lem:contractingCones} and Theorem~\ref{thm:qinvMeas}, assume  $f\in \cC_{\sqrt{\gm}}$, and let  $h_n=\frac{\nu_{\sg^{-n}\om}(q_{\sg^{-n}\om}) f}{q_{\sg^{-n}\om}}$.
Then, recalling that $\Einf(q_{\sg^{-n}\om})=1$, we get that
$\|h_n\|_\infty \leq \|q_{\sg^{-n}\om} \|_\infty \|f\|_{\infty}$.
Also, 
$h_n q_{\sg^{-n}\om}=\nu_{\sg^{-n}\om}(q_{\sg^{-n}\om})f \in \cC_{\sqrt{\gm}}\subset \cC_{\sqrt{\gm}a_{\sg^{-n}\om}}$.
Recalling \eqref{eq:lamnu}, and writing the RHS of \eqref{eq:doc3} with the choice $(h,f)=(h_n,1)$,  yields, as in \eqref{eq:doc4},
\begin{equation}\label{eq:docMET}
	\begin{split}
		&\frac{\big \| \nothing\cL_{\sg^{-n}\om}^{(n)}(h_n q_{\sg^{-n}\om})- \Einf(\cL_{\sg^{-n}\om}^{(n)}(h_n q_{\sg^{-n}\om})) q_{\om} \big\|_\infty } {\lm^{-,(n)}_{\sg^{-n}\om}\nu_{\om}(q_{\om})}\\
		&=\lt(\lm^{(n)}_{\sg^{-n}\om} \rt)^{-1}  \big\| \nothing\cL_{\sg^{-n}\om}^{(n)}(f)- \Einf(\cL_{\sg^{-n}\om}^{(n)}(f)) q_{\om} \big\|_\infty \leq C_\om  \|h_n\|_{\infty}  \vartheta^n
		\leq C_\om  \|q_{\sg^{-n}\om} \|_\infty \vartheta^n  \|f\|_{\infty}  .
	\end{split}
\end{equation}

Observe that
$$(\lm^{(n)}_{\sg^{-n}\om} )^{-1} \Einf(\cL_{\sg^{-n}\om}^{(n)}(f)) q_{\om} -\nu_{\sg^{-n}\om}(f) \psi_{\om} = \frac{ \Einf(\cL_{\sg^{-n}\om}^{(n)}(f))}{\lm^{(n)}_{\sg^{-n}\om}} 
\nu_\om\lt(q_{\om} -\frac{\cL_{\sg^{-n}\om}^{(n)}(f)}{ \Einf(\cL_{\sg^{-n}\om}^{(n)}(f))}\rt)
\psi_{\om}.$$
Note also that for any $r>\vartheta$ there exists $D_\om>0$ such that
$\lt\|q_{\om} -\frac{\cL_{\sg^{-n}\om}^{(n)}(f)}{ \Einf(\cL_{\sg^{-n}\om}^{(n)}(f))} \rt\|_\infty\leq D_\om r^n$, by \eqref{eq:Cauchyfn}.
Recalling that 
$\lm^{(n)}_{\sg^{-n}\om}=\nu_\om(\cL_{\sg^{-n}\om}^{(n)}(1))\geq \Einf (\cL_{\sg^{-n}\om}^{(n)}(1))$, we get
\begin{equation}\label{eq:docMET2}
	\|(\lm^{(n)}_{\sg^{-n}\om} )^{-1}  \Einf(\cL_{\sg^{-n}\om}^{(n)}(f)) q_{\om} -\nu_{\sg^{-n}\om}(f) \psi_{\om}\|_\infty \leq D_\om \| \psi_{\om}\|_\infty r^n  \|f\|_\infty.
\end{equation}
The triangle inequality applied to \eqref{eq:docMET} and \eqref{eq:docMET2}, combined with \eqref{eq:Qn}, shows that
$\lim_{n\to\infty} \frac{1}{n}\log \|\nothing Q_{\sg^{-n}\om}^{(n)}f\|_\infty<0$.

Since the limit in Lemma~\ref{lem:eqdens}\ref{it:eqdens1} satisfies $q_\om^f\in \BV$, then 
$\lim_{n\to\infty} \frac{1}{n}\log \|\cL_{\sg^{-n}\om}^{(n)}f\|_{\BV}=\lim_{n\to\infty} \frac{1}{n}\log \Einf(\cL_{\sg^{-n}\om}^{(n)}f)=\lim_{n\to\infty} \frac{1}{n}\log \|\cL_{\sg^{-n}\om}^{(n)}f\|_{\infty}$. Thus, \eqref{eq:Qn} and the previous paragraph yield, for every $f\in \cC_{\sqrt{\gm}}$,
\begin{equation}\label{eq:growthBV}
	\lim_{n\to\infty} \frac{1}{n}\log \|Q_{\sg^{-n}\om}^{(n)}f\|_{\BV}=\lim_{n\to\infty} \frac{1}{n}\log \|Q_{\sg^{-n}\om}^{(n)}f\|_\infty<0.
\end{equation} 
Since every $f\in \BV$ may be written as $f=f_1-f_2$ such that $f_i\in \cC_{\sqrt{\gm}}$, and the growth rate of a sum is bounded above by the largest of the terms' growth rates, then $\lim_{n\to\infty} \frac{1}{n}\log \|  Q_{\sg^{-n}\om}^{(n)} f\|_{\infty}<0$ holds for every $f\in \BV$.
Thus, $\lim_{n\to\infty} \frac{1}{n}\log \|Q_{\sg^{-n}\om}^{(n)}\|_{\BV}<0$.

Finally, Kingman's sub-additive ergodic theorem implies that $\lim_{n\to\infty} \frac{1}{n}\log \|Q_{\om}^{(n)}\|_{\BV}=\lim_{n\to\infty} \frac{1}{n}\log \|\nothing Q_{\sg^{-n}\om}^{(n)}\|_{\BV}$, so $\lim_{n\to\infty} \frac{1}{n}\log \|Q_{\om}^{(n)}\|_{\BV}<0$, as claimed.
In fact, our arguments show that $\lim_{n\to\infty} \frac{1}{n}\log \|Q_{\om}^{(n)}\|_{\BV}\leq \log \vartheta$, for any $\vartheta>\tanh(\tfrac{D}{4})$, as in Lemma~\ref{lem:coneCont}. 

The multiplicative ergodic theorem \cite{FLQ2} ensures 
uniqueness of a (measurable) equivariant splitting, which in the present context translates into
uniqueness of the tuple $(\psi_\om, \nu_\om, \lm_\om)$. 
Furthermore, the theorem shows that  
$
\Lm_1=\int \log\lm_\om dm = \lim_{n\to \infty}\frac1n \log\|\cL_\om^{(n)}\|_{\BV},
$
for \maeom.
\end{proof}

\section{Examples} \label{S:examples1d}
\subsection{Sufficient conditions for strongly contracting potentials}
In this section, we present conditions to ensure a random potential is strongly contracting.
Assume 
\[\log \#\mathring{\cZ}_{\om}, \log \|g_\om\|_\infty, \log   \Einf(g_{\om}),  \frac{\var(g_{\om})}{\|g_{\om}\|_\infty} \in L^1(m).\]
Since 
$1/\Einf(g_{\om,n})$ and $1/b_{\om,f}^{(n)}$ are sub-multiplicative, Kingman's subadditive ergodic theorem implies that the following limits exist and are $m$-a.e. constant,
\begin{align*}
-\phi^-&:=\lim  \frac{1}{n}\log (1/\Einf_{X_{\om,n}}(g_{\om}^{(n)})),\quad
\bt_f:=\lim  \frac{1}{n}\log b_{\om,f}^{(n)}.
\end{align*}
In addition, they coincide with the limits of the, respectively, decreasing and increasing sequences 
\begin{align*}
\Big (-\phi^{-}_n&:= \int -\frac{1}{n}\log \Einf_{X_{\om,n}}(g_{\om}^{(n)}) dm \Big)_{n\in\NN}, \quad
\Big(\bt_{f,n}:=\int  \frac{1}{n}\log b_{\om,f}^{(n)} dm \Big)_{n\in\NN}.
\end{align*}
Furthermore, $\|g_\om\|_\infty^{(n)}$ is multiplicative, so by Birkhoff's ergodic theorem, the limit 
$\phi^+:=\lim   \frac{1}{n}\log \|g_\om\|_\infty^{(n)}$ 
exists, and is $m$-a.e. equal to  $\int \log \|g_\om\|_\infty dm$.
Recalling that $\tilde{S}_{n,\om}(g)=\sum_{j=0}^{n-1} \frac{\var(g_{\sg^j\om})}{\|g_{\sg^j\om}\|_\infty}$, Birkhoff's ergodic theorem implies $\lim   \frac{1}{n}\log (1+ \tilde{S}_{n,\om}(g))=0$.

The following bound on $\xi_\om^{(n)}$ may be considered a random generalization of \cite[Lemma 6.3]{liverani_maume_2003}; see \cite[Proposition 15.3]{AFGTV2} for a proof.

\begin{proposition}[\cite{AFGTV2}]\label{prop:nonfull}
The following inequality holds for $\xi_{\om}^{(n)}$, the largest number of contiguous
non-full intervals for $T_\om^{(n)}$:
\[
\xi_{\om}^{(n)}
\leq n \prod_{j=0}^{n-1} (\xi^{(1)}_{\sg^j\om}+2).
\]
\end{proposition}

Synthetizing the previous discussion, we get the following.

\begin{example}\label{ex:contPots}
Assume 
$\log \#\mathring{\cZ}_{\om}, \log \|g_\om\|_\infty, \log   \Einf(g_{\om}),  \frac{\var(g_{\om})}{\|g_{\om}\|_\infty} \in L^1(m).$ Then,
$\{\log g_\om\}$ is a random  strongly contracting potential for the random (open or closed) map $\{(T_\om, H_\om)\}$ if any of the following conditions hold:
\begin{enumerate}
	\item  
	\label{it:nstar=1} Case $n_*=1$: 
	$$\int \log \|g_\om\|_\infty -\log \Einf(g_\om)
	+\log (3) +\log \lt(1+ \frac{\var(g_{\om})}{\|g_{\om}\|_\infty}\rt) +\log (1+2\xi^{(1)}_{\om})-\log b_{\om,f} dm <0.$$
	
	\item \label{it:genEx}  Either $\int \log \|g_\om\|_\infty- \log\Einf(g_\om) + \log (2+\xi_{\om}^{(1)})- \log b_{\om,f} dm <0;$
	or, slightly more generally, 
	\begin{equation*}\label{eq:contPot}
		\int \log \|g_\om\|_\infty dm-\phi^- +\int \log (2+\xi^{(1)}_{\om})dm - \bt_{f} <0.
	\end{equation*}
	
	\item \label{it:xibd}  There exist  $K,\xi\geq 1$ such that $ \xi_\om^{(n)}\leq K\xi^n$ for \maeom \ and every $n\in \NN$, and
	\begin{equation*}
		\int \log \|g_\om\|_\infty - \log\Einf(g_\om) dm+\log \xi - \bt_{f} <0.
	\end{equation*}
\end{enumerate}
\end{example}

\begin{remark}
Roughly speaking, Example~\ref{ex:contPots}\eqref{it:nstar=1}  corresponds to having, on average, potentials with small logarithmic amplitude and controlled variation, and open maps with few contiguous non-full branches and lots of full branches. For constant potentials with no (pair of) contiguous non-full branches, this condition simplifies to 
$\int \log b_{\om,f} 
dm >\log (9).$
\end{remark}
\begin{remark}\label{rmk:ssv}
Example~\ref{ex:contPots}\eqref{it:xibd} allows us to compare our setting with the one-dimensional setting of \cite{stadlbauer_suzuki_varandas_2020}, which deals with $C^1$ potentials $\phi_\om=\log g_\om$ and $C^1$ local diffeomorphisms $T_\om$ satisfying a condition called (P).  In that setting, the maps do not have discontinuities, so $\xi_\om^{(n)}=0$, and the condition in Example~\ref{ex:contPots}\eqref{it:xibd}  reduces to 
$ \int \|\phi_\om\|_\infty - \Einf(\phi_\om) dm-\bt_{f} <0$. Condition (P) 
may be written as
$\int \|\phi_\om\|_\infty - \Einf(\phi_\om) + \log (1+ \|D\phi_\om\|_\infty \diam(\IO))   dm<- \int \log\frac{A_\om}{b_{\om,f}} dm$, where, in the notation of \cite{stadlbauer_suzuki_varandas_2020}, $A_\om=\sg_\om^{-1}p_\om + L_\om q_\om\geq 1$.
Since $\bt_f\geq \int \log b_{\om,f} dm$, the notion of strongly contracting potential is more general than condition (P) in this case.
\end{remark}

\subsection{Non-transitive systems and a covering criterion}\label{S:nonCovering}
The following example shows that our results are applicable to non-transitive systems. 
\begin{example}\label{ex:nonmixing}
Consider interval maps $T_\om:\IO\to \IO$ as in Figure~\ref{fig:ncmap}, where the (possibly empty) left interval of the hole $H_\om$ is positioned within the given branch. Then, $b_{\om,f}\in \{5, 6\}$, $\xi_\om^{(1)}=2$ 
and
$\int \log (2+\xi_{\om}^{(1)})- \log b_{\om,f} dm
\leq \log 4 - \log 5 <0$.
Thus, 
Example~\ref{ex:contPots}\eqref{it:genEx} ensures the constant potential $\log g_\om=0$ is strongly contracting, provided  $\log \|T'_\om \|_\infty,
\log   \Einf|T'_{\om}|,  \frac{\var(|T'_{\om}|)}{\|T'_{\om}\|_\infty} \in L^1(m)$.
In this case, it also follows from Definition~\ref{def:cp} that  
$-t\log |T'_\om|$ is strongly contracting for sufficiently small $t>0$.
\end{example}

\begin{remark}
The map of Figure~\ref{fig:ncmap} is not topologically transitive.
In fact, when the $T_\om$ have a (common) Markov partition, the corresponding transition matrices have a (non-random) absorbing set corresponding to the branches within the invariant interval around 1/2.
\end{remark}

\begin{remark} If a map $T_\om$ has an invariant interval $J\subsetneq \IO$, as in Figure~\ref{fig:ncmap}, and $g_\om=1/|T'_\om|$, then 
\[
\log \|g_\om\|_\infty  +\log (2+\xi^{(1)}_{\om}) \geq 0 \quad \text{and}
\quad
\log \Einf(g_\om)+\log b_{\om,f} dm <0.
\]
Indeed, the first inequality comes from two facts: (i) if $N$ is the number of monotonic branches of $T_\om|_J$, then $N\leq 2+\xi^{(1)}_{\om}$, as all except for possibly the leftmost and rightmost branches of the invariant interval are non-full; and (ii) $\Einf_{x\in J} |T_\om'(x)| \leq N$. The second inequality follows from $\Esup_{x\in \IO} |T_\om'(x)| > {b_{\om, f}}$.

In particular, if all maps $\{T_\om\}$ have a common invariant interval, then
the geometric potential $\{-\log |T'_\om|\}$ is not strongly contracting. This is in agreement with the fact that such a system has at least one non-fully supported random invariant measure absolutely continuous with respect to Lebesgue measure.
\end{remark}

To show a stronger result in this direction, we introduce a notion of covering in the random (closed) setting, due to Buzzi \cite{BuzziEDC}, and
show it is satisfied in wide generality, provided the potential $-\log |T_\om'|$ is strongly contracting.

\begin{definition}
A random  map $\{T_\om\}$ is called \emph{covering} if for every open interval $J\subset \IO$, 
there exists $M_\om(J)\in\NN$ such that
\begin{align}\label{eq: def of open covering}
	\Einf \cL_{\om}^{(M_\om(J))}\ind_J(x) > 0.
\end{align}
In the context of this work, \eqref{eq: def of open covering} is equivalent to $T_\om^{(M_\om(J))}(J)=\IO$.
\end{definition}

\begin{lemma}\label{lem:contracting4geompot}
Consider a random map $\{T_\om\}$ and assume the random potential
$-\log |T_\om'|$ is strongly contracting. Furthermore, assume
$\Om$ is a Borel subset of a separable complete metric space, $m$ is a Borel probability and $\sg$ is an homeomorphism.
Then, $\{T_\om\}$ is covering.
\end{lemma}

\begin{proof}
Let $\leb$ denote the normalized Lebesgue measure on $\IO$.
A simple but crucial observation is that in this case $\nu_\om(f)=\int f d\leb$, where $\nu_\om$ is as in \S\ref{S:confMeas}. 
Indeed,
$\int \cL_\om f d\leb=\int f d\leb$ holds by the change of variables formula and hence $f\mapsto \int f d\leb$ is an equivariant functional (in fact it is invariant by all $\cL_\om$, and $\lm_\om^+=1$).
Theorem~\ref{thm:randompf} ensures uniqueness of the equivariant conformal measure, so 
$\nu_\om(f)=\int f d\leb$.

Now we show the random map is covering. Let $J\subset \IO$ be an open interval.
Then, $0<\leb(J)=\nu_\om(\ind_J)=\lim_{n\to\infty}\frac{\Einf \cL_{\om}^{(n)} \ind_J}{\Einf\cL_{\om}^{(n)} 1}$. In particular, there exists $M>0$ such that $\Einf \cL_{\om}^{(M)} \ind_J>0$,
as needed.
\end{proof}

\subsection{Random intermittent maps}\label{S:intermittent}
For $0<\gm<1$, consider the Manneville--Pomeau map $f_\gm: [0,1] \to [0,1]$, given by
$$f_\gm(x)=\begin{cases} x(1+2^\gm x^\gm) \quad &0\leq x < \frac12,\\ 2x-1 \quad &\frac12\leq x \leq 1. \end{cases}$$
This is a class of intermittent maps, with a neutral fixed point at 0,
which have been investigated as a model of non-uniformly hyperbolic behaviour since the work of Liverani, Saussol and Vaienti \cite{liverani_probabilistic_1999}.
More recently, Demers and Todd have investigated open and closed intermittent maps with geometric potentials $-t \log |f'_\gm|$ in \cite{DemersTodd-SlowFastEscape-17}.
The next example shows a family of strongly contracting geometric potentials for random intermittent maps.
\begin{example}\label{ex:MPmaps}
For $j=1,2,\dots$, let $\gm_j \in (0,1)$.
Let $\Om=\cup_{j=1}^\infty \Om_j$ be an (at most) countable partition of $\Om$ into measurable sets, and for each $\om \in \Om_j$, let $T_\om=f_{\gm_j}$.
Let $0\leq t<\frac{\log2}{\log 3}\approx 0.63$.
Then, the geometric potential $\{\log g_\om:=-t \log |T'_\om|\}$ is strongly contracting for $\{T_\om\}$.
Indeed, we note that for all $0<\gm<1$, we have $\Einf |f_\gm'|=1$ and $\|f_\gm'\|_\infty < 3$.
Furthermore, $\xi^{(n)}_\om=0, b^{(n)}_{\om,f}=2^n$ for all $n\in \NN$.
Thus,  Example~\ref{ex:contPots}\eqref{it:xibd} (with $K=\xi=1$) yields the claim, since $\var(\log|T'_\om|)\in L^1(m)$ and
\[ \int \log \|g_\om\|_\infty - \log\Einf(g_\om) dm+\log \xi - \bt_{f} \leq 0+t\log 3+0-\log2 <0.\]
\end{example}

The following example treats random intermittent maps with holes. 
\begin{example}\label{ex:intMapsHoles}
Let $\Om=\cup_{j=1}^\infty \Om_j$ be an (at most) countable partition of $\Om$ into measurable sets, and for each $\om \in \Om_j$, let $T_\om=T_j:\IO:=[0,1] \to [0,1]$ be a piecewise smooth map with a hole $H_\om=H_j$ satisfying the following conditions:
\begin{enumerate}[(i)]
	\item $T_\om(0)=0$ and $T'_\om(0)=1=\Einf_{\IO}|T'_\om|$,
	\item $\|T'_\om\|_\infty\leq K_\om$, with $\log K_\om \in L^1(m)$,
	\item $\var(\log|T'_\om|)\leq v_\om$, with $v_\om \in L^1(m)$,
	\item $(T_\om, H_\om)$ has at most two contiguous non-full branches; for instance, this happens if $T_\om$ only has full branches and $H_\om$ consists of a single interval, and
	\item $(T_\om, H_\om)$ has $b_{\om,f}$ full branches, and
	$\bt:= \int \log b_{\om,f} dm > \log 4 + t_0 \int \log K_\om dm$, for some $0\leq t_0<1$.\footnote{Note that $K_\om \geq b_{\om, f}$.}
\end{enumerate}
Then, for every $0\leq t \leq t_0$,
the geometric potential $\{\log g_\om:=-t \log |T'_\om|\}$ is strongly contracting for $\{(T_\om, H_\om)\}$.
Indeed, Example~\ref{ex:contPots}\eqref{it:genEx} yields the claim, since
\[ \int \log \|g_\om\|_\infty - \log\Einf(g_\om) + \log (2+\xi_{\om}^{(1)}) - \log b_{\om,f} dm\leq 0+t\log K_\om dm + \log 4 - \bt<0.\]
\end{example}

\subsection{Random open systems and escape rates}\label{S:escapeRates}
The following example, similar to \cite[\S13]{AFGTV2}, relates the maximal Lyapunov exponent of open and closed systems to the escape rate of a conformal measure through the holes. 
\begin{example}
Assume  $\{\log g_\om\}$ is a strongly contracting potential for the random closed map $\{T_\om\}$.
Assume  $(H_\om^\ep)_{0<\ep\leq\ep_0}$ is an increasing family of holes for each $\om\in \Om$. That is, 
$H_\om^\ep$ is a finite union of intervals, and $\emptyset:=H_\om^0\subset H_\om^{\ep'} \subset H_\om^{\ep}$ for $\ep'<\ep$. Let 
$b_{\om,f}^{\ep}$ the number of full branches of $\{(T_\om,H_\om^\ep)\}$ and $\xi_{\om}^{\ep}$ the largest number of contiguous
non-full intervals for $\{(T_\om,H_\om^\ep)\}$.
Suppose there exist $b_\om, \xi_\om>0$ such that for every $\ep\geq 0$, 
$b^\ep_{\om,f}\geq b_\om$ and $\xi^\ep_\om\leq \xi_\om$,
and
assume
\[\int \log \|g_\om\|_\infty- \log\Einf(g_\om) + \log (2+\xi_{\om})- \log b_{\om} dm <0.\]
Then, for each $0<\ep\leq\ep_0$, $\{\log g_\om\}$ is a strongly contracting potential for the random open map $\{(T_\om,H_\om^\ep)\}$.
Let $\nu^\ep_\om$ and $q^\ep_\om$ be the conformal measures and equivariant densities from Theorem~\ref{thm:qinvMeas}, respectively, and let $\Lm^\ep$ the maximal Lyapunov exponent (expected pressure). Then $\ep\mapsto \Lm^\ep$ is non-increasing. Indeed,  
if $\ep'<\ep$, because of the monotonicity of the holes, for every $\om\in \Om, n\in \NN$, we have
$\Einf(\cL_{\om}^{\ep',(n)}1)\geq \Einf(\cL_{\om}^{\ep,(n)}1)$.
Since $\Lm^\ep=\lim_{n\to\infty} \frac1n\log\Einf(\cL_{\om}^{\ep,(n)}1)$, $\ep\mapsto \Lm^\ep$ is non-increasing.

Furthermore, for $0\leq \ep'<\ep$,
$\Lm^{\ep'}-\Lm^\ep$ gives the escape rate of the measure $\nu^{\ep'}$ through $\{H_\om^\ep\}$.
That is, $-\lim\frac1n\log \nu_\om^{\ep'}(X^\ep_{\om,n})=\Lm^{\ep'}-\Lm^\ep$, where $X^\ep_{\om,n}$ is the $n-$step survivor set for $\{(T_\om,H_\om^\ep)\}$. Indeed,
\begin{align*}
	\nu_{\om}^{\ep'}(X^\ep_{\om,n-1})
	&=
	\frac{1}{\lm_{\om}^{\ep',(n)}}\nu^{\ep'}_{\sg^n(\om)}\lt(\cL_{\om}^{\ep',(n)}(\ind_{X^\ep_{\om,n-1}})\rt)
	=\frac{1}{\lm_{\om}^{\ep',(n)}}\nu^{\ep'}_{\sg^n(\om)}\lt(\cL_{\om}^{\ep,(n)} 1\rt)
	\\
	&=\frac{\Einf(\cL_{\om}^{\ep,(n)} 1)}{\lm_{\om}^{\ep',(n)}} \lt(\nu^{\ep'}_{\sg^n(\om)}(q^\ep_{\sg^n(\om)})-\nu^{\ep'}_{\sg^n(\om)}\lt(\frac{\cL_{\om}^{\ep,(n)}1}{\Einf(\cL_{\om}^{\ep,(n)} 1)}-q^\ep_{\sg^n(\om)}\rt)\rt).
\end{align*}
Lemma~\ref{lem:eqdens} implies that $\lim_{n\to\infty}\frac1n\log \lt\|\frac{\cL_{\om}^{\ep,(n)}1}{\Einf(\cL_{\om}^{\ep,(n)} 1)}-q^\ep_{\sg^n(\om)}\rt\|_\infty<0$.
Since $\nu^{\ep'}_{\sg^n(\om)}$ is a probability measure, $\Einf q^\ep_{\sg^n(\om)}=1$ and $\|q^\ep_{\sg^n\om}\|_\infty$ is tempered, then $\lim_{n\to\infty}\frac1n\log\nu^{\ep'}_{\sg^n(\om)}(q^\ep_{\sg^n(\om)})=0$.
Thus, 
$$\lim_{n\to\infty}\frac1n\log \nu_\om^{\ep'}(X^\ep_{\om,n})=
\lim_{n\to\infty}\frac1n\log  \Einf(\cL_{\om}^{\ep,(n)} 1) - \lim_{n\to\infty}\frac1n \log\lm^{\ep',(n)}_\om= \Lm^\ep-\Lm^{\ep'},$$ as claimed.
\end{example}

\section*{Acknowledgments}
The authors are partially supported by an ARC Discovery Project.

\bibliographystyle{abbrv}
\bibliography{infvar_refs}
\end{document}